\documentclass{article}

\usepackage{amsmath,amssymb,amsthm}
\usepackage[colorlinks=true]{hyperref}
\usepackage{pdfsync}
\usepackage[dvips]{graphicx}
\usepackage{appendix}

\usepackage[applemac]{inputenc}
 
\def\R{\mathbb{R}}
\def\N{\mathbb{N}}

\newtheorem{thm}{Theorem}
\newtheorem{lem}{Lemma}
\newtheorem{prop}{Proposition}
\newtheorem{cor}{Corollary}
\theoremstyle{definition}\newtheorem{rem}{Remark}
\theoremstyle{definition}
\theoremstyle{definition}\newtheorem{defi}{Definition}

\renewcommand{\d}{\displaystyle}

\renewcommand{\geq}{\geqslant}
\renewcommand{\leq}{\leqslant}

\newcommand{\MTTP}{{\bf (MTTP)}}

\title{Planar tilting maneuver of a spacecraft: singular arcs in the minimum time problem and chattering}

\author{Jiamin Zhu\footnote{Sorbonne Universit\'es, UPMC Univ Paris 06, CNRS UMR 7598, Laboratoire Jacques-Louis Lions, F-75005, Paris, France (\texttt{zhu@ann.jussieu.fr}).}
\and
Emmanuel Tr\'elat\footnote{Sorbonne Universit\'es, UPMC Univ Paris 06, CNRS UMR 7598, Laboratoire Jacques-Louis Lions, Institut Universitaire de France, F-75005, Paris, France (\texttt{emmanuel.trelat@upmc.fr}).}
\and
Max Cerf\footnote{Airbus Defence and Space, Flight Control Unit, 66 route de Verneuil, BP 3002, 78133 Les Mureaux Cedex, France (\texttt{max.cerf@astrium.eads.net}).}
}

\date{}

\begin{document}

\maketitle

\begin{abstract}
In this paper, we study the minimum time planar tilting maneuver of a spacecraft, from the theoretical as well as from the numerical point of view, with a particular focus on the chattering phenomenon. We prove that there exist optimal chattering arcs when a singular junction occurs. Our study is based on the Pontryagin Maximum Principle and on results by M.I. Zelikin and V.F. Borisov. We give sufficient conditions on the initial values under which the optimal solutions do not contain any singular arc, and are bang-bang with a finite number of switchings. Moreover, we implement sub-optimal strategies by replacing the chattering control with a fixed number of piecewise constant controls. Numerical simulations illustrate our results.
\end{abstract}

\vspace{0.5cm}


\tableofcontents

\section{Introduction} \label{sec_Intro}
The minimum time planar tilting problem of a spacecraft consists of controlling the spacecraft, with certain prescribed terminal conditions on the attitude angles, accelerations, and the velocity direction, while minimizing the maneuver time and keeping constant the yaw and rotation angles. This problem is of interest for (at least) two reasons. The first one is that the resulting optimal strategy can be used during the rocket ascent phase, along which the attitude and the orbit motions are strongly coupled. The second one is that, as we will prove in this paper, the optimal trajectories, solutions of the problem, exhibit a chattering phenomenon which is, in itself, difficult and thus interesting to analyze, but which is also rather a bad news in view of practical issues. We thus analyze it in detail, providing sufficient conditions on the terminal conditions under which the optimal strategy does not involve any chattering, and in case chattering occurs, we provide alternative sub-optimal strategies.

\subsection{The optimal control problem} \label{sec_ProbStat}
Let us formulate the minimum time planar tilting maneuver problem (pitching movement of the spacecraft). Throughout the paper, we restrict our study to the planar case, in the sense that the spacecraft movement remains in a plane.

\paragraph{Model.}
We assume that the Earth is a fixed ball in the inertial space, and that the velocity of the wind is zero. We consider an axial symmetric spacecraft (see Figure \ref{planar_model}). Taking coordinates $(x,y)$, we adopt the following notations:
\begin{itemize}
\item $v_x$ and $v_y$ are the velocity components of the velocity vector $\vec{v}$;
\item $\theta$ is the pitch angle of the spacecraft;
\item $\omega$ is the angular velocity with respect to the Earth;
\item $r>0$ is the distance between the spacecraft mass center $O_b$ and the center $O$ of the Earth;
\item $\ell>0$ is the distance from the thrust point $P$ to the mass center of the spacecraft $O_b$;
\item $\vec{e}_a$ is the unit vector along the symmetric axis of the spacecraft, and $\vec{e}_c$ is the unit vector perpendicular to $\vec{e}_a$ pointing to the North;
\item $I$ is the moment of inertia along the $\vec{e}_a \times \vec{e}_c$ axis;
\item $\mu$ is the angle between the thrust vector $\vec{T}$ and the symmetric axis $\vec{e}_a$ of the spacecraft, and we must have $|\mu| \leq \mu_{max}$;
\item $\gamma$ is the flight path angle defined as the angle between the velocity $\vec{v}$ and the axis $\vec{x}$.
\end{itemize}

\begin{figure}[h]
\centering
	 \includegraphics[scale = 0.9]{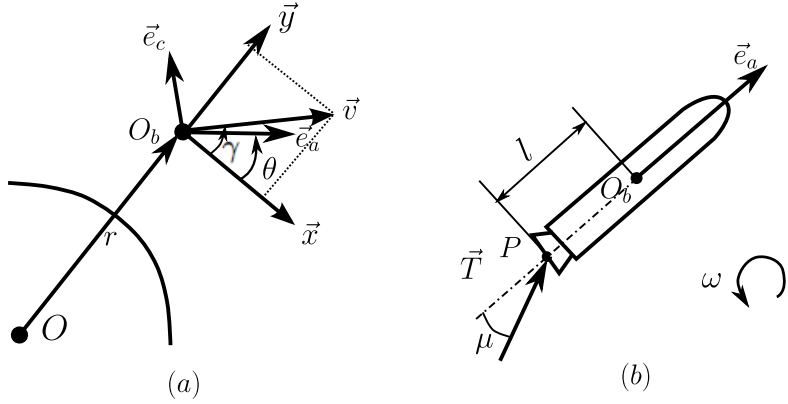}
	\caption{ Frames and parameters in problem $\MTTP$.}
	\label{planar_model}
\end{figure}

The motion of the spacecraft is controlled by the angle $\mu$. Since $\mu$ is small in practice (between $\pm 10$ degrees), we assume that $cos \mu \approx 1$ and $sin \mu \approx \mu$. Under this small angle assumption, the spacecraft evolves in time according to the system
\begin{equation}\label{sys1}
\begin{split}
	\dot{v}_x&= a \cos \theta  - c v_x v_y , \\
	\dot{v}_y&= a \sin \theta + c v_x^2 - g_0, \\
	\dot{\theta} &= \omega - c v_x, \\
	\dot{\omega}&= b u ,
\end{split}
\end{equation}
with control $u=-\mu/\mu_{max} \in [-1,1]$ and $\d{a=T/m}$, $\d{c=1/r}$, $\d{b=T \ell \mu_{max}/I}$ being positive constants.

Actually, in our numerical simulations, we will use the parameters of Ariane 5 launchers (see Table \ref{sim_param}).
The modulus of the velocity $v = \sqrt{v_x^2+v_y^2}$ takes values in $[0,v_{m}]$, and for the pitch angle and the angular velocity we have the estimate $\vert\theta\vert\leq \theta_{\max}$ and $\vert\omega\vert\leq \omega_{\max}$. 

\begin{table}[h] 
\centering
\begin{tabular}{|c|c|c|c|c|c|c|}
  \hline
     & $a$ & $b$ & $c$ & $v_{m}\ (m/s)$ & $\omega_{\max}\ (rad/s)$ & $\theta_{\max} (rad)$\\
  \hline
   Value  & $12$ & $0.02$ & $1\times 10^{-6}$ & $5000$ & $0.3$ & $\pi$\\
  \hline
\end{tabular}
\caption{System parameters.}
\label{sim_param}
\end{table}

In the sequel, for convenience, we set $x_1=v_x$, $x_2=v_y$, $x_3=\theta$ and $x_4=\omega$. Denoting by $x=(x_1,x_2,x_3,x_4)$, the system \eqref{sys1} can be written as the \emph{single-input control-affine system}
\begin{equation}\label{singleinputcontrolaffinesystem}
	\dot{x} = f_0 (x) + u f_1(x) ,
\end{equation}
where $f_0$ and $f_1$ are the smooth vector fields on $\R^4$ defined by
\begin{equation}\label{def_f0f1}
f_0 = ( a \cos x_3  - c x_1 x_2 ) \frac{\partial}{\partial x_1}
	     + (a \sin x_3 + c x_1^2- g_0 ) \frac{\partial}{\partial x_2}
	     + (x_4 - c x_1) \frac{\partial}{\partial x_3} ,\qquad
f_1 = b \frac{\partial}{\partial x_4} .
\end{equation}

\paragraph{Terminal conditions.}
The requirements are the following:
\begin{itemize}
\item all initial variables are fixed;
\item the final values of the variables $\theta$ and $\omega$ are prescribed, and we require that, at the final time $t_f$ (which is let free), the velocity vector $\vec{v}(t_f)$ be parallel to the spacecraft axis $\vec{e}_a(t_f)$. 
\end{itemize}
It is indeed natural to consider $\vec{v}(t_f) \parallel \vec{e}_a(t_f)$ as a terminal condition, because the spacecraft considered is of rocket-type, and such spacecrafts are usually planned to maintain a small angle of attack along the flight. Note that, here, the angle of attack is the angle between the spacecraft axis $\vec{e}_a$ and the velocity $\vec{v}$. The zero angle of attack condition ensures that the aerolift is null in order to avoid excessive loading of the structure (see \cite{BLAKELOCK}).

Since $\gamma = \arctan (x_2/x_1)$ and $v = \sqrt{x_1^2 + x_2^2}$, we have 
\begin{equation} \label{flightangle}
	\dot{\gamma} = (a \sin (x_3 - \gamma)- g_0 \cos \gamma )/ v+c v \cos \gamma ,\qquad
	\dot{v} = a \cos (x_3 - \gamma) - g_0 \sin \gamma .
\end{equation}
The final condition above is then written as $\gamma(t_f) = x_{3}(t_f)$. In term of $v$ and $\gamma$, the velocity components $x_1$ and $x_2$ are $x_{1}=v \cos \gamma$ and $x_{2}=v \sin \gamma$. We set $v(0)=v_0$ and $\gamma(0)=\gamma_0$.

\paragraph{Minimum time planar tilting problem.}
Let $x_0 \in \R^4$, and let $v_0$, $\gamma_0$, $x_{30}$, $x_{40}$ and $x_{3f}$ be real numbers.
In terms of the variables $x=(x_1,x_2,x_3,x_4)$, the initial point is defined by
$$
x_0 =(x_{10},x_{20},x_{30},x_{40}),
$$
with $x_{10}=v_0 \cos \gamma_0$ and $x_{20}=v_0 \sin \gamma_0$, and the final target is the submanifold of $\R^4$ defined by
\begin{equation*}\label{} 
M_1 = \{ (x_1,x_2,x_3,x_4)\in\R^4 \mid x_2\cos x_{3f} - x_1 \sin x_{3f} = 0,\ x_3 = x_{3f},\ x_4=0 \}.
\end{equation*}
Throughout the paper, we consider the optimal control problem, denoted in short $\MTTP$, of steering the control system \eqref{sys1} from $x(0)=x_0$ to the final target $M_1$ in minimal time $t_f$, under the control constraint $u(t) \in [-1,1]$.

\subsection{State of the art}
The minimum time spacecraft attitude maneuver problem has been widely studied (see, e.g., \cite{Bilimoria,Fleming,Proulx,Shen}). Besides, there are many works on the coupled attitude orbit problem (see, e.g., \cite{Gong,Knutson,Wang}) and on the minimum time orbit transfer (see, e.g., \cite{Caillau,BFT,Kim,Thorne,Yue}). 

The problem $\MTTP$ under consideration in this paper is however more related to the well-known \emph{Markov-Dubins problem} (in short, MD problem) and to variants of it.
Indeed, if the system were to be directly controlled by the variable $x_3$, then, by taking the target manifold to be a single point ($x(t_f)=x_f$) and letting $a=b=1$, $c=0$, $g_0=0$, the system \eqref{sys1} would be written as
$$
\dot{x}_1= \cos x_3, \quad
\dot{x}_2= \sin x_3, \quad
\dot{x}_3= u, 
$$
and therefore, the problem $\MTTP$ coincides with the MD problem, which was first settled in \cite{Markov} and was analyzed in detail by Dubins and many others (see, e.g., \cite{Dubins,Reeds,Sussmann}). It has been shown that the optimal strategy for the MD problem consists in first reaching the singular arc with a single bang arc, then, in following this singular arc until one is sufficiently close to the final target, and finally, in leaving the singular arc in order to reach the target with a single bang arc.

If we assume that $g_0 \neq 0$, i.e., if we have the system
$$
\dot{x}_1= \cos x_3, \quad
\dot{x}_2= \sin x_3 - g_0, \quad
\dot{x}_3= u, 
$$
then the problem $\MTTP$ coincides with the \emph{Zermelo-Markov-Dubins problem} (in short, ZMD problem) with constant wind field $(w_x,w_y)=(0,-g_0)$ (see, e.g., \cite{Bakolas,McGee,Glizer,Techy}). The optimal strategy of this problem consists of a finite number of bang and singular arcs. Both the MD and the ZMD problems may involve a singular arc because the singular controls of these problems are of intrinsic order one (see further in the present paper for this notion). 

However, this is not the case for the problem $\MTTP$ for which the singular control is of intrinsic order two. In this sense, a problem closer to $\MTTP$ (with $a=b=1$, $c=0$, $g_0=0$) is the \emph{Markov-Dubins problem with angular acceleration control} (in short, MDPAAC) (see \cite{Laumond,Sussmann2}). In that problem, the model is a dynamic extension of the MD system, given by
$$
\dot{x}_1= \cos x_3, \quad
\dot{x}_2= \sin x_3 , \quad
\dot{x}_3= x_4,\quad
\dot{x}_4=u, 
$$
The existence of a chattering phenomenon for MDPAAC was first put in evidence in \cite{Sussmann2}. Although the optimality status of these chattering arcs remains unclear, the discussion of the chattering phenomenon brings interesting issues for the analysis of the present problem $\MTTP$.

The system we consider here can also be seen as a variation of the MD system, with nonconstant wind and controlled by the inertial control. Thus, we expect the solution of the problem $\MTTP$ to share properties similar to MDPAAC (in particular, chattering), MD and ZMD (in view of the global behavior of the solution). 

In fact, using \cite{Zelikin1}, we will be able to prove the existence and the optimality of the chattering phenomenon in the problem $\MTTP$. The chattering phenomenon (also occuring in MDPAAC) is caused by singular controls of intrinsic order two. It makes the optimal synthesis for the problem $\MTTP$ essentially different from that of the MD or ZMD problem. 

However, in some sense the optimal solution of problem $\MTTP$ consists as well of three pieces: the first piece consists of bang arcs to reach the singular arc, the second piece is a singular arc, and the third piece consists of a succession of bang arcs finally reaching the target submanifold. 

Since the chattering phenomenon causes difficulties in practical use, we will also provide sufficient conditions on the terminal conditions, under which the chattering arcs do not appear in the optimal solution. This prediction result will be useful in order to decide which numerical method (either direct, or indirect, or sub-optimal) is the most appropriate.

\subsection{Chattering phenomenon} \label{sec_chat}
Let us recall that we speak of a \textit{chattering phenomenon} (sometimes also called a Fuller's phenomenon), when the optimal control switchings an infinite number of times over a compact time interval. It is well known that, if the optimal trajectory of a given optimal control problem involves a singular arc of higher order, then no connection with a bang arc is possible and then bang arcs asymptotically joining the singular arc must chatter. On Figure \ref{chattering}(b), the control is singular over $(t_1,t_2)$, and the control $u(t)$ with $t \in (t_1-\epsilon_1,t_1) \cup (t_2,t_2+\epsilon_2)$, $\epsilon_1>0$, $\epsilon_2>0$ is chattering. The corresponding optimal trajectory is called a chattering trajectory. On Figure \ref{chattering}(a), the chattering trajectory ``oscillates'' around the singular part and finally ``gets off" the singular trajectory with an infinite number of switchings. 

In this paper, we call \textit{singular junction}, the junction point between a singular arc and a non-singular arc.

\begin{figure}[h]
\centering
	\includegraphics[scale=0.9]{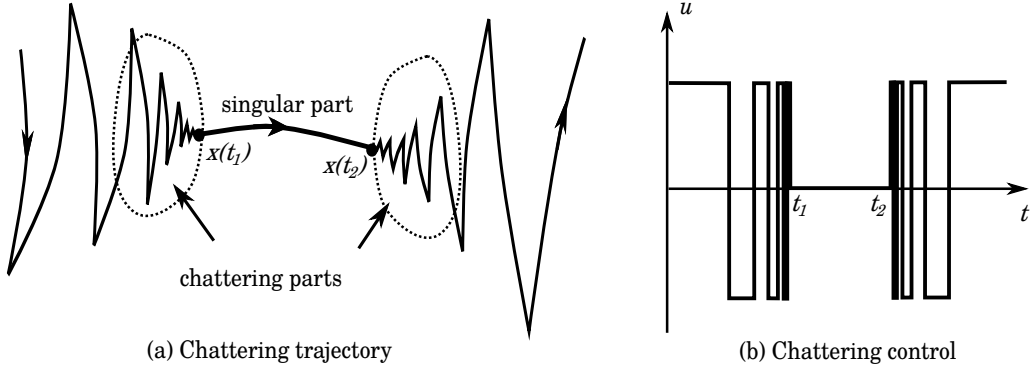}
	\caption{An illustration of chattering phenomenon.}
	\label{chattering}
\end{figure}

To better explain the chattering phenomenon, we recall the well-known Fuller problem (see  \cite{Fuller,MARCHAL}), which is the optimal control problem
\begin{equation*}
\left\{
\begin{split}
& \min \int_0^{t_f} x_1(t)^2 \, dt , \\
& \dot{x}_1(t)=x_2(t),\  \dot{x}_2(t)=u(t), \quad |u(t)| \leq 1,\\ 
& x_1(0)=x_{10},\ x_2(0)=x_{20},\ x_1(t_f)=0,\ x_2(t_f)=0,\quad t_f\ \textrm{free}.
\end{split} 
\right.
\end{equation*}
We define $ \xi = \left( \frac{\sqrt{33}-1}{24} \right)^{1/2}$ as the unique positive root of the equation $\d{\xi^4+\xi^2/12-1/18=0}$, and we define the sets
\begin{equation*}
\begin{split}
\Gamma_{+}&=\{ (x_1,x_2)\in \R^2 \mid x_1= \xi x_2^2,\ x_2<0 \} , \quad\ \ 
R_{+}=\{ (x_1,x_2)\in \R^2 \mid x_1 <  - \mathrm{sign} (x_2) \xi x_2^2 \} , \\
\Gamma_{-}&=\{ (x_1,x_2)\in \R^2 \mid x_1= - \xi x_2^2,\ x_2>0 \} , \quad
R_{-}=\{ (x_1,x_2)\in \R^2 \mid x_1 >  - \mathrm{sign} (x_2) \xi x_2^2 \} .
\end{split}
\end{equation*}
Then the optimal synthesis of the Fuller problem is the following (see \cite{Fuller2,SCHATTLER,Wonham}).
The optimal control is given in feedback form by
\begin{equation*}
	u^{\ast}=\begin{cases}
		\phantom{-}1 & \textrm{if}\ x \in R_{+} \bigcup \Gamma_{+} , \\
		-1& \textrm{if}\ x \in R_{-} \bigcup \Gamma_{-}  .
	  \end{cases}
\end{equation*}
The control switchings from $u=1$ to $u=-1$ at points on $\Gamma_{-}$ and from $u=-1$ to $u=1$ at points on $\Gamma_{+}$. The corresponding trajectories crossing the switching curves $\Gamma_{\pm}$ transversally are chattering arcs with an infinite number of switchings that accumulate with a geometric progression at the final time $t_f>0$.

The optimal synthesis for the Fuller problem is drawn on Figure \ref{Fuller}. 
The solutions of the Fuller problem are chattering solutions since they switch transversally on the switching curves $\Gamma_{\pm}$ until finally reaching the target point on the singular surface defined by the union of all singular solutions.

\begin{figure}[h]
\centering
	\includegraphics[scale=0.4]{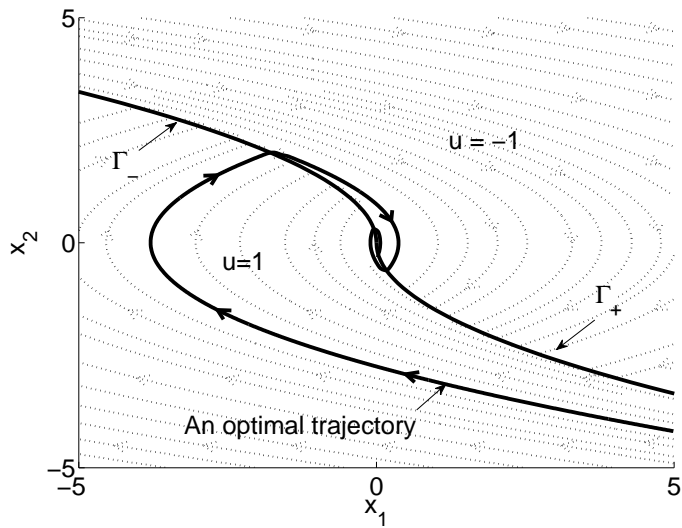}
	\caption{Optimal synthesis for the Fuller problem.}
	\label{Fuller}
\end{figure}

In fact, the optimal control of the Fuller problem, denoted as $u^{\ast}$, contains a countable set of switchings of the form
\begin{equation*}
	u^{\ast}(t)=\begin{cases}
		\phantom{-}1 & \textrm{if}\ t \in [t_{2k},t_{2k+1}), \\
		-1& \textrm{if}\ t \in [t_{2k+1},t_{2k+2}] ,
	  \end{cases}
\end{equation*}
where $\{ t_{k} \}_{k \in \mathbb{N}}$ is a set of switching times that satisfies ${ (t_{i+2} - t_{i+1}) < (t_{i+1} - t_{i})}$, $i \in \N$ and converges to $t_f < +\infty$. This means that the chattering arcs contain an infinite number of switchings within a finite time interval $t_f>0$.

\medskip

The analysis of chattering arcs is challenging. Based on a careful analysis of the Fuller problem, M.I. Zelikin and V.F. Borisov obtained a geometric portrait of solutions in the vicinity of the second order singular solutions (see \cite{Zelikin1,Zelikin2}). These solutions are called chattering solutions. Using their results, we will be able to prove rigorously the existence and optimality of chattering solutions in our problem $\MTTP$. 

The basic idea of their approach to provide sufficient conditions for optimality is based on the following well-known sufficient optimality condition:
\begin{quote}
\emph {Let $M$ be a smooth manifold of dimension $n$, and let $T^*M$ be its cotangent bundle, endowed with its canonical symplectic structure. If a submanifold $L$ of $T^*M$ generated by a given Hamiltonian system on $T^*M$ is Lagrangian, then a ``nice'' regular projection of trajectories of $L$ onto $M$ can also be seen, by canonical injection, as a Lagrangian submanifold of $T^*M$, and the trajectories are locally optimal in $C^0$ topology.} 
\end{quote}

Recall that a submanifold $L$ of a smooth manifold $M$ is said to be Lagrangian if ${\oint_{\gamma} p \, dx =0}$ for every piecewise smooth closed contour $\gamma$ on the manifold. Hence, the manifold consisting of the solutions of a Hamiltonian system with transversality condition ($p\, dx=0$ on the target manifold) is Lagrangian. Denote the cost functional to be minimized as $C(\cdot,\cdot)$. A trajectory $\bar{x}(\cdot)$ is said to be locally optimal in $C^0$ topology if, for every neighborhood $V$ of $\bar{x}(\cdot)$ in the state space, for every real number $\eta$ so that $| \eta | \leq \epsilon$, for every trajectory $x(\cdot)$, associated to a control $v$ on $[0,T+\eta]$, contained in $W$, and satisfying $x(0) = \bar{x}(0) = x_0$, $x(T+\eta) = \bar{x}(T)$, there holds $C(T+\eta,v) \geq C(T,u)$.

Hence, the problem of proving the local optimality of a solution comes down to constructing a Lagrangian submanifold. The usual way to construct a Lagrangian submanifold is to integrate backward in time the Hamiltonian system from the target point. However, this is not applicable for the chattering arcs because the control is not anymore piecewise constant and the length of switching intervals goes to zero at the singular junction. 

In order to overcome this flaw of the usual approach, M.I. Zelikin and V.F. Borisov proposed an explicit procedure to construct Lagrangian submanifolds filled by chattering trajectories. The main difficulty of this construction procedure is to analyze the regularity of the projections of the extremal lifts to the state space. 

\medskip

When using numerical methods to solve an optimal control problem, the occurrence of  chattering arcs may be an obstacle to convergence. Recall that there are two main types of numerical methods for solving optimal control problems: indirect methods and direct methods
(see, e.g., the survey paper \cite{Trelat}).

The direct methods (see \cite{Betts}) consist of discretizing the state and the control and thus of reducing the problem to a nonlinear optimization problem (nonlinear programming) with constraints. Using standard optimization routines, it is then possible to make converge the algorithm for the Fuller problem. Of course, the numerical solution which is obtained can only have a finite number of switchings, because in the approximation scheme, the chattering control is actually approximated with a piecewise constant control. 

The indirect methods consist of numerically solving a boundary value problem obtained by applying the Pontryagin Maximum Principle, by means of a shooting method. An indirect method is also called a shooting method (see \cite{StoerBulirsch}). In \cite{Bonnans}, it is shown that the presence of chattering arcs may imply ill-posedness (non-invertible Jacobian) of shooting methods for single-input problems. According to \cite{Zelikin2}, the difficulty is due to the numerical integration of the discontinuous Hamiltonian system (i.e., the right-hand side of the Hamiltonian system is discontinuous) because the chattering solutions worsen the approximation and error estimates during calculation for the standard numerical integration methods.

\subsection{Structure of the paper}
The paper is structured as follows.

In Sections \ref{sec_PMP} and \ref{sec_CompSingArcs}, the Pontryagin Maximum Principle (PMP) and an usual way to compute singular controls are recalled. Section \ref{sec_geomchatter} is devoted to recall some results of \cite{Zelikin1,Zelikin2}, explaining geometric features of the chattering phenomenon, based on a semi-canonical form of the Hamiltonian system along singular extremals of order two, with the objective of showing how these theoretical results can be applied in practice.

The non-singular (bang-bang) extremals of $\MTTP$ are analyzed in Section \ref{sec_ReguArcs}, and the Lie bracket configuration is given in \ref{sec_Lie}. We prove in Section \ref{sec_SingArcs} that the singular controls for $\MTTP$ are of intrinsic order two, which implies the existence of chattering arcs.  Based on the results of M.I. Zelikin and V.F. Borisov, we prove in Section \ref{sec_ChatArcs} that the optimal chattering arcs of the problem $\MTTP$ are locally optimal in $C^0$ topology. 

In Section \ref{sec_SingPred}, we provide, for the cases with $c=0$ and $c >0$ respectively, sufficient conditions on the initial values under which the optimal solutions do not contain any singular arc, and do not chatter. Numerical simulations, in Section \ref{sec_NumeChatPred}, illustrate these conditions.

Since chattering is not desirable in view of practical issues, we propose some sub-optimal strategies in Section \ref{sec_SuboSolu}, by approximating the chattering control with piecewise constant controls. Our numerical results provide evidence of the convergence of sub-optimal solutions to optimal solutions (but this convergence is not analyzed from the theoretical point of view in the present paper).

\section{Geometric analysis of chattering} 
Let $M$ be a smooth manifold of dimension $n$, and let $M_1$ be a submanifold of $M$. We consider on $M$ the minimal time control problem
\begin{equation} \label{pb_ocp}
\left\{ \begin{split}
	& \min t_f , \\
	& \dot{x}(t) = f_0(x(t))+u(t) f_1(x(t)),\quad |u(t)| \leq 1 , \\
	& x(0) = x_0,\ x(t_f) \in M_1 , \quad t_f\geq 0\ \textrm{free},
\end{split}\right.
\end{equation}
where $f_0$ and $f_1$ are two smooth vector fields on $M$. Since the system and the instantaneous cost are control-affine, and the control constraint is compact and convex, according to classical results (see, e.g., \cite{Cesari,Trelatbook}), there exists at least one optimal solution $(x(\cdot),u(\cdot))$, defined on $[0,t_f]$.

\subsection{Application of the Pontryagin maximum principle} \label{sec_PMP}
According to the Pontryagin maximum principle (in short, PMP, see \cite{PONTRYAGIN}), there must exist an absolutely continuous mapping $p(\cdot)$ defined on $[0,t_f]$ (called adjoint vector), such that $p(t)\in T^*_{x(t)}M$ for every $t\in[0,t_f]$, and a real number $p^0 \leq 0$, with $(p(\cdot),p^0)\neq 0$, such that
\begin{equation} \label{Hamiltonsys}
	\dot{x}(t) = \frac{\partial H}{\partial p}(x(t),p(t),p^0,u(t)),\quad
	\dot{p}(t) = -\frac{\partial H}{\partial x}(x(t),p(t),p^0,u(t)) ,
\end{equation}	
almost everywhere on $[0,t_f]$, where
\begin{equation} \label{Hamiltonianfun}
H(x,p,p^0,u) = \langle p, f_0(x)\rangle+u \langle p,f_1(x) \rangle + p^0 
\end{equation}
is the Hamiltonian of the optimal control problem \eqref{pb_ocp}, and (the final time $t_f$ being free)
\begin{equation} \label{Hamiltonsys2}
H(x(t),p(t),p^0,u(t)) = \max_{-1\leq v(t)\leq 1} H(x(t),p(t),p^0,v(t)),
\end{equation}
almost everywhere on $[0,t_f]$.
Moreover, we have the transversality condition
\begin{equation} \label{Hamiltonsys3}
	p(t_f) \perp T_{x(t_f)} M_1 ,
\end{equation}
where $T_{x(t_f)}M_1$ denotes the tangent space to $M_1$ at the point $x(t_f)$.

The quadruple $(x(\cdot),p(\cdot),p^0,u(\cdot))$ is called the extremal lift of $x(\cdot)$.
An extremal is said to be normal (resp., abnormal) if $p^0 < 0$ (resp., $p^0 = 0$). 

We define the functions
\begin{equation} \label{HamilFun}
h_0(x,p)=\langle p, f_0(x) \rangle  ,\quad
h_1(x,p)=\frac{\partial H}{\partial u}(x,p,p^0,u) = \langle p, f_1(x) \rangle .
\end{equation}
It follows from \eqref{Hamiltonsys2} that $u(t) = \mathrm{sign}(\varphi(t))$, whenever $\varphi(t) = h_1(x(t),p(t))\neq 0$.
For this reason, the function $\varphi$ is also called the switching function. 

\paragraph{Bang arcs.}
We say that the trajectory $x(\cdot)$ restricted to a sub-interval $I$ of $[0,t_f]$ is a \emph{bang arc} if $u(t)$ is constant along $I$, equal either to $+1$ or to $-1$. We say that the trajectory is bang-bang on $[0,t_f]$ if it is the concatenation of bang arcs.

\paragraph{Singular arcs.}
If $\varphi(t)=h_1(x(t),p(t))=0$ along a sub-interval $I$ of $[0,t_f]$, then the relation \eqref{Hamiltonsys2} does not allow to directly infer the control, and in that case we speak of a \emph{singular arc}, or of a \emph{singular extremal}.

\medskip

Equivalently, a singular control is defined as follows.
The end-point mapping $E: \R^n \times \R \times L^\infty(0,+\infty;\R) \to \R^n$ of the system is defined by $E(x_0,t_f,u)=x(x_0,t_f,u)$ where $t \mapsto x(x_0,t,u)$ is the trajectory solution of the control system, corresponding to the control $u$, such that $x(x_0,0,u)=x_0$ (the domain of definition is then the set of controls for which the trajetory is indeed globally defined on $[0,t_f]$).
A trajectory $x(\cdot)$, defined on $[0,t_f]$, with $x(0)=x_0$, associated with a control $u$, is said to be \emph{singular} if the differential $\partial_u E(x_0,t_f,u)$ is not of full rank. Accordingly, we speak of a \emph{singular control}.
It is well known that a trajectory $x(\cdot)$ is singular on $[0,t_f]$ if and only if it has an extremal lift $(x(\cdot),p(\cdot),p^0,u(\cdot))$, satisfying \eqref{Hamiltonsys} and $h_1(x(t),p(t))=0$ on $[0,t_f]$ (see \cite{BonnardChyba,Trelatbook}). This extremal lift is called a \emph{singular extremal}.

\subsection{Computation of singular arcs} \label{sec_CompSingArcs}
In order to compute a singular control, the usual method (see \cite{BonnardChyba}) consists of differentiating repeatedly the relation 
\begin{equation} \label{singularcond1}
\varphi(t)=h_1(x(t),p(t))=0
\end{equation}
with respect to time, until the control appears in a nontrivial way. Using the Hamiltonian system \eqref{Hamiltonsys}, such derivations are done thanks to Poisson brackets and Lie brackets. By differentiating \eqref{singularcond1} a first time (along the interval $I$), we obtain
\begin{equation} \label{singularcond2}
0 = \dot\varphi(t) = \{h_0,h_1\}(x(t),p(t)) = \langle p(t),[f_0,f_1](x(t))\rangle,
\end{equation}
which is a new constraint. Differentiating a second time, we obtain
\begin{equation*}
\begin{split}
0 = \ddot\varphi(t) 
& =\{h_0,\{h_0,h_1\}(x(t),p(t)) + u(t) \{h_1,\{h_0,h_1\}(x(t),p(t)) \\
&= \langle p(t),[f_0,[f_0,f_1](x(t))\rangle + u(t) \langle p(t),[f_1,[f_0,f_1](x(t))\rangle,
\end{split}
\end{equation*}
in which the control now appears in a nontrivial way provided that $\{h_1,\{h_0,h_1\}\}(x(t),p(t)) < 0$. The latter condition is known as \emph{strengthened Legendre-Clebsch condition}. Under this condition, we can indeed compute the singular control as
$$
u(t) = -\frac{\{ h_0,\{ h_0,h_1\}\}(x(t),p(t))}{\{ h_1,\{ h_0,h_1\}\}(x(t),p(t))}.
$$
It can be noted that the first derivative of $\varphi(\cdot)$ does not make appear the control. Hence, two derivations in time are at least necessary in order to make appear the control in a nontrivial way. Such controls are also said to be of \emph{minimal order}, and actually this property is generic (see \cite{Bon-Kup97,Chitour_Jean_Trelat}). Hereafter, due to the fact that optimal singular arcs have to appear with an even number of derivations, we also say that such singular arcs are of \emph{intrinsic order one}.

If $\{h_1,\{h_0,h_1\}\}(x(t),p(t))= 0$ identically on $I$, then the above computation does not suffice and we need to differentiate more.
In that case, we see that we have two additional constraints:
\begin{equation} \label{singularcond3}
\{h_0,\{h_0,h_1\}\}(x(t),p(t)) = \langle p(t),[f_0,[f_0,f_1]](x(t))\rangle = 0,
\end{equation}
and
\begin{equation*}
\{h_1,\{h_0,h_1\}\}(x(t),p(t)) = \langle p(t),[f_1,[f_0,f_1]](x(t)) = 0,
\end{equation*}
for every $t\in I$. 

Let us recall the concept of the order of a singular control. Roughly speaking, it is the first integer $m$ such that the control $u$ appears in a nontrivial way in the $(2m)^\textrm{th}$-derivative of the switching function $\varphi(\cdot)$ (see \cite{SCHATTLER,Zelikin2}). 

\begin{defi} \label{def_singularorder}
The singular control $u$ (along the sub-interval $I$) is said to be of \emph{local order} $k$ if the conditions 
\begin{equation*} 
\frac{\partial}{\partial u} \varphi^{(i)}(x(t),p(t)) =  0, \quad i=0,1,\cdots,2k-1, \quad
\frac{\partial}{\partial u} \varphi^{(2k)}(x(t),p(t)) \neq  0,
\end{equation*}
hold along the sub-interval $I$. 
If moreover the Lie brackets $[f_1,[\mathrm{ad}^{i}f_0.f_1]]$, $i=0,\cdots,2k-2$, are identically equal to zero (over the whole space), then the singular control $u$ is said to be of \emph{intrinsic order} $k$.
\end{defi}

We adopt the usual notations $\mathrm{ad}f_0.f_1 = [f_0,f_1]$ (resp., $\mathrm{ad}h_0.h_1 = \{h_0,h_1\}$) and $\mathrm{ad}^if_0.f_1 = [f_0,\mathrm{ad}^{i-1}f_0.f_1]$ (resp., $\mathrm{ad}^ih_0.h_1 = \{h_0,\mathrm{ad}^{i-1}h_0.h_1\}$).

\begin{rem}
If a singular control $u$ is of \emph{local} order two, then the conditions (along $I$)
$$ 
\frac{\partial}{\partial u} \varphi^{(2)}(t) =\langle p(t),[f_1,\mathrm{ad} f_0.f_1] (x(t))\rangle= 0,
$$
and 
$$ 
\frac{\partial}{\partial u} \varphi^{(3)}(t) =2 \langle p(t),[f_1,\mathrm{ad}^2 f_0.f_1] (x(t))\rangle + u(t) \langle p(t),[f_1,[f_1,\mathrm{ad} f_0.f_1]] (x(t))\rangle= 0,
$$ 
are additional constraints that must be satisfied along the singular arc. In contrast, if $u$ is of \emph{intrinsic} order two, then these conditions are trivially satisfied since $[f_1,\mathrm{ad} f_0.f_1]\equiv0$ and $[f_1,\mathrm{ad}^2 f_0.f_1]\equiv0$. In the present paper, we are in the situation of singular arcs of intrinsic order two, and we will then focus on that case.
\end{rem}

Actually, we did not consider, in the above definition, the case where the first nonzero derivative is of odd order. Indeed, such singular controls are actually never optimal, and hence we do not consider them in our analysis. This fact is due to the following well-known result, usually referred to as Kelley's condition for singular extremals of local order $k$ (see \cite{Kelley,Krener}):
\begin{quote}
\textit{
If a trajectory $x(\cdot)$, associated with a singular control $u(\cdot)$, is locally time-optimal on $[0,t_f]$ in $L^\infty$ topology, then the generalized Legendre-Clebsch condition
$$
(-1)^k \frac{\partial }{\partial u} \frac{d^{2k} h_1}{dt^{2k}}  \leq 0,
$$
is satisfied along the extremal. 
Recall that a trajectory $\bar{x}(\cdot)$ is said to be locally optimal in $L^\infty$ topology if, for every neighborhood $V$ of $u$ in $L^\infty([0,T+\epsilon],U)$, for every real number $\eta$ so that $| \eta | \leq \epsilon$, for every control $v \in V$ satisfying $E(x_0,T+\eta,v) = E(x_0,T,u)$ 
there holds $C(T+\eta,v) \geq C(T,u)$, where $E: \R^n \times \R \times L^\infty(0,+\infty;\R) \to \R^n$ is the end-point mapping defined by $E(x_0,t_f,u)=x(x_0,t_f,u)$.
}
\end{quote}
Therefore, the generalized Legendre-Clebsch condition for a singular control of \emph{local} order $2$ is
\begin{equation*}
\langle p(t),[f_1,\mathrm{ad}^3f_0.f_1] (x(t)) + [f_0,[f_0,[f_1,[f_0,f_1]]]] (x(t)) + [f_0,[f_1,\mathrm{ad}^2f_0.f_1]] (x(t))\rangle \leq 0,
\end{equation*}
and if the singular control of \emph{intrinsic} order $2$, then this condition takes the simpler form
$$
\langle p(t),[f_1,\mathrm{ad}^3f_0.f_1] (x(t))\rangle \leq 0.
$$

\medskip

Turning back to the previous computation, if the singular control is of intrinsic order two, then by differentiating $\ddot{\varphi}(t) = \{h_0,\{h_0,h_1\}\}(x(t),p(t))$, we get
\begin{equation*}
\begin{split}
0 = \varphi^{(3)}(t)
& =\{h_0,\mathrm{ad}^2h_0.h_1\}(x(t),p(t)) + u(t) \{h_1,\mathrm{ad}^2h_0.h_1\}(x(t),p(t)) \\
&= \langle p(t), [f_0,\mathrm{ad}^2f_0.f_1](x(t))\rangle + u(t) \langle p(t), [f_1,\mathrm{ad}^2f_0.f_1](x(t))\rangle,
\end{split}
\end{equation*}
which, using the fact that $[f_1,\mathrm{ad}^2f_0.f_1]\equiv 0$, leads to the additional constraint
\begin{equation} \label{singularcond4}
\{h_0,\mathrm{ad}^2h_0.h_1\}(x(t),p(t)) = \langle p(t), [f_0,\mathrm{ad}^2f_0.f_1](x(t))\rangle = 0.
\end{equation}
Differentiating again, we get
\begin{equation*}
\begin{split}
0 = \varphi^{(4)}(t)
& =\{h_0,\mathrm{ad}^3 h_0.h_1\}(x(t),p(t)) + u(t) \{h_1,\mathrm{ad}^3 h_0.h_1\}(x(t),p(t)) \\
&= \langle p(t), [f_0,\mathrm{ad}^3 f_0.f_1](x(t))\rangle + u(t) \langle p(t), [f_1,\mathrm{ad}^3 f_0.f_1](x(t))\rangle.
\end{split}
\end{equation*}
By definition, we have $\langle p(t), [f_1,\mathrm{ad}f_0^3.f_1](x(t))\rangle \neq 0$, and thus the singular control is 
\begin{equation} \label{singularcontrolorder2}
u(t) = -\frac{ \mathrm{ad}^4 h_0.h_1 (x(t),p(t))}{\{h_1, \mathrm{ad}^3 h_0.h_1 \}(x(t),p(t))},
\end{equation}
which is smooth. 

\begin{rem}
Along such a singular arc of intrinsic order two, the singular control is given by \eqref{singularcontrolorder2} and the constraints \eqref{singularcond1}, \eqref{singularcond2}, \eqref{singularcond3}, \eqref{singularcond4} must be satisfied along the arc.
\end{rem}

\medskip

In this paper, we are actually concerned with optimal singular trajectories of intrinsic order two, which cause the occurrence of a chattering phenomenon in our problem. Let us recall the following result (see \cite{Kelley,McDanell,Zelikin1}).

\begin{lem}\label{thm_bs}
We assume that the optimal solution $x(\cdot)$ of the optimal control problem \eqref{pb_ocp} involves a singular arc (on a sub-interval $I$) of intrinsic order two, for which the strengthened generalized Legendre-Clebsch condition
$$
\frac{\partial }{\partial u} \frac{d^{4} h_1(t)}{dt^{4}} = \{h_1,\mathrm{ad}^3h_0.h_1 \}(x(t),p(t))  < 0
$$
holds true along an extremal lift. If we have $\vert u(t)\vert < 1$ along the singular arc, then the singular arc cannot be matched directly with any bang arc. In particular, if $I$ is a proper subset of $[0,t_f]$, then the optimal solution chatters, in the sense that there is an infinite number of bang arcs accumulating at the junction with the singular arc.
\end{lem}

Although this result is known, we will provide a short proof of it when analyzing our spacecraft problem in Section \ref{sec_SingArcs}.

\begin{rem}
Note that the Fuller problem can be adapted to fit in the framework above, although this is not a minimum time problem. Actually, it suffices to add the objective as a third state variable $x_3$, evolving according to $\dot{x}_3 = x_1^2/2$, and then the Fuller problem can be interpreted, by uniqueness of the solution, as a minimum time problem with the vector fields $f_0(x)=(x_2,0,x_1^2/2)^\top$ and $f_1=(0,1,0)^\top$. The corresponding singular extremal is therefore given by $u =0$, $x_1=x_2=p_1=p_2=p^0=0$ and $p_3 <0$ being constant. The solutions of the Fuller problem are optimal abnormal extremals for this three-dimensional problem. Moreover, it is easy to see that $u=0$ is a singular control of intrinsic order two, along which the strengthened generalized Legendre-Clebsch condition is satisfied ($p_3 <0$). Then Lemma \ref{thm_bs} can be applied. 
\end{rem}

\subsection{Geometric analysis of the chattering phenomenon}\label{sec_geomchatter}
In this section, we recall some results on chattering solutions established in \cite{Zelikin1,Zelikin2}. Since these references are not always easy to read, our objective is also to provide a more pedagogical exposition of these results and to show how they can be used in practice.

Recall that a chattering solution is the optimal solution corresponding to the chattering control which switches an infinite number of times over a compact time interval.

\subsubsection{Semi-canonical form}
The semi-canonical form (see \cite{Kuppa,Zelikin1}) is a way of writing the Hamiltonian system \eqref{Hamiltonsys} in a neighborhood of its singular arcs, which will be used later to analyze the solutions near (in $C^0$ topology) singular arcs of intrinsic or local order two. The main idea is to design a variable change that leads to a form involving the switching function and its derivatives directly as variables. This makes the analysis of the extremals near the singular arcs more convenient. 

Let $x(\cdot)$ be an optimal trajectory of \eqref{pb_ocp} on $[0,t_f]$, and let $(x(\cdot),p(\cdot),p^0,u(\cdot))$ be an extremal lift (coming from the PMP). We assume that $x(\cdot)$ involves a singular arc of intrinsic second order two, along the sub-interval $I$, satisfying the strengthened generalized Legendre-Clebsch condition.

The Hamiltonian \eqref{Hamiltonianfun} can be rewritten as $H=h_0+u h_1+p^0$, with $h_0$ and $h_1$ defined by \eqref{HamilFun}. We assume that
\begin{equation} \label{funindpcond1}
\dim \mathrm{Span}\{f_1, \mathrm{ad}f_0.f_1, \mathrm{ad}^2f_0.f_1, \mathrm{ad}^3f_0.f_1\} = 4 .
\end{equation}
We define the new coordinates
\begin{equation} \label{XPtoZW}
z_1= h_1,\quad
z_2= h^{(1)}_1= \{ h_0,h_1\},\quad
z_3= h^{(2)}_1= \mathrm{ad}^2 h_0.h_1,\quad
z_4= h_1^{(3)}= \mathrm{ad}^3 h_0.h_1 ,
\end{equation}
and using that $[f_1,[f_0,f_1]] \equiv 0$ and that $\{h_1,\mathrm{ad}^3h_0.h_1\}<0$ along $I$, we have
\begin{equation*} 
\dot{z}_1=z_2 ,\quad
\dot{z}_2=z_3 ,\quad
\dot{z}_3=z_4 ,\quad
\dot{z}_4=\alpha(x,p)+  u \beta(x,p) ,
\end{equation*}
where $\alpha = \mathrm{ad}^4 h_0.h_1$ and $\beta=\{h_1,\mathrm{ad}^3h_0.h_1\} < 0$.

Note that $z_1$ is chosen as the switching function $\varphi(t)=h_1(x(t),p(t))$ and $z_i$ is chosen as the $(i-1)$-th derivative of the switching function. In fact, using that $[f_1,[f_0,f_1]] \equiv 0$ and using Jacobi's identity, we have 
$$
\{h_1,\{h_0,\{h_0,h_1\}\}\} = -\{h_0,\{\{h_0,h_1\},h_1\}\} - \{\{h_0,h_1\},\{h_1,h_0\}\} = \{h_0,\{h_1,\{h_0,h_1\}\}\} \equiv 0 .
$$ 
This, together with $\beta < 0$, indicates that the singular control considered here is of intrinsic order two and satisfies the generalized Legendre-Clebsch condition. By definition, we have $z_i=0$, $i=1,2,3,4$, along such a singular arc.

From \eqref{funindpcond1}, we infer that $z_1$, $z_2$, $z_3$, $z_4$ are functionally independent in the neighborhood of the extremal lift $(x(\cdot),p(\cdot))$, along $[0,t_f]$. We complement $z=(z_1,z_2,z_3,z_4)$ with $w=(w_1,\cdots,w_{2n-4}) \in \R^{2n-4}$ such that the Jacobi matrix of the mapping $ (x,p) \mapsto (z,w)$ is nondegenerate, i.e.,
\begin{equation*} 
	\det \left( \frac{D(z,w)}{D(x,p)} \right ) \neq  0 ,
\end{equation*}
along the extremal.
Since our point of view is local, we assume that $(x,p)$ and $(z,w)$ live in $\R^{2n}$.
The Hamiltonian system \eqref{Hamiltonsys} can be rewritten, locally along the extremal, as
\begin{equation} \label{Hamsys}
\dot{z}_1=z_2 ,\quad
\dot{z}_2=z_3 ,\quad
\dot{z}_3=z_4 ,\quad
\dot{z}_4=\alpha(z,w)+  u \beta(z,w),\quad
\dot{w}=F(z,w,u) ,
\end{equation}
and the extremal control is given by
\begin{equation*}
u(t)=\begin{cases}
1 & \textrm{if}\ z_1(t)>0 ,\\
-\alpha/\beta & \textrm{if}\ z_1(t)=0 , \\
-1 & \textrm{if}\ z_1(t)<0 .
\end{cases}
\end{equation*}
Accordingly, we define the \emph{singular surface} (smooth manifold consisting of singular extremals of second order) as
\begin{equation*} 
	S = \{(z,w) \mid (z_1,z_2,z_3,z_4)=(0,0,0,0)\},
\end{equation*}
and the \emph{switching surface} as 
\begin{equation*} 
	\Gamma = \{(z,w) \mid z_1=0\}.
\end{equation*}
If a trajectory $z(\cdot)$ is a solution of \eqref{Hamsys}, then a straightforward calculation yields that $z_{\lambda}=G_{\lambda}(z(t/\lambda))$ is also a solution of \eqref{Hamsys}, for any number $\lambda>0$, where
\begin{equation} \label{mappingG}
G_{\lambda} (z(\frac{t}{\lambda})) = \left( \lambda^4 z_1\left(\frac{t}{\lambda}\right),\lambda^3 z_2\left(\frac{t}{\lambda}\right),\lambda^2 z_3\left(\frac{t}{\lambda}\right),\lambda z_4\left(\frac{t}{\lambda}\right) \right).
\end{equation} 
This is an important property for the Fuller problem (self-similar solutions).

The system \eqref{Hamsys} is useful in order to analyze the qualitative behavior of solutions near the singular surface consisting of singular extremals of intrinsic order two. To include some Hamiltonian systems having singular arcs of local order two, we consider a small perturbation of the system \eqref{Hamsys} in the neighborhood of a given point $(0,w_0) \in S$, given by
\begin{equation} \label{small_perturbation}
\left\{\begin{array}{l}
        \d{\dot{z_1} = z_2 + f_1(z,w,u)},\\
        \d{\dot{z_2} = z_3 + f_2(z,w,u)},\\
        \d{\dot{z_3} = z_4 + f_3(z,w,u)},\\
        \d{\dot{z_4}= \alpha(w) + u \beta(w) + f_4(z,w,u)},\\
        \d{\dot{w}= F(z,w,u) },
\end{array}\right.
\end{equation}
with $f_i(z,w,u)=\mathrm{o}(z_{i+1})$, i.e.,
\begin{equation} \label{sp_cond}
\lim_{\lambda \rightarrow 0^+} \lambda^{-(5-i)} \vert f_i(G_{\lambda}(z(t/\lambda)),w,u) \vert < +\infty, \quad i=1,2,3,4.
\end{equation}
The system \eqref{small_perturbation}-\eqref{sp_cond} is called a \emph{semi-canonical form}.

\begin{rem}
The variables $(z,w)$ can be chosen differently from \eqref{XPtoZW} in order to get a simpler local system \eqref{small_perturbation}. This is why this form is called semi-canonical, and not canonical. Moreover, this change of variable is not unique.
\end{rem}


\subsubsection{Geometry of chattering extremals}
The first result concerns the existence of chattering solutions. In contrast to Lemma \ref{thm_bs}, this result can also be applied to the case of singular arcs of \emph{local} order two, and it describes the phase portrait of optimal extremals in the vicinity of a manifold of singular arcs of order two. 

Recall that the singular surface $S$ for the system \eqref{small_perturbation} is of codimension $4$. The surface $S$ satisfies four constraints $z_1=0$, $z_2=0$, $z_3=0$, $z_4=0$ corresponding respectively to null derivatives of the switching functions $\varphi^{(i)}$, $i=0,1,2,3$. Considering a point $(0,w_0) \in S$, if $\beta(w_0)<0$ and $| \alpha(w_0) | < -\beta(w_0)$, there exists a neighborhood of this point in which the singular extremals passing through it satisfy the generalized Legendre-Clebsch condition and the singular control $|u|=|-\alpha(w) / \beta(w)|<1$ is admissible. The following proposition indicates that, for any point in such a neighborhood, there exists a family of chattering extremals coming into this point, and there is another family of chattering extremals emanating from this point. Note that a family of chattering extremals is a one-parameter family, with the parameter $\lambda$ defined in \eqref{mappingG}.

\begin{prop}[Bundles with chattering fibers] \label{thm1}
Consider the system \eqref{small_perturbation}, in an open neighborhood of the point $(0,w_0)$. If $\beta(w_0)<0$ and $|\alpha(w_0)|< - \beta(w_0)$, then there exists an open neighborhood $\mathcal{O}$ of $w_0$ in $\R^{2n-4}$ such that, for any $w \in \mathcal{O}$, there are two one-parameter families of chattering extremals intersecting only at the point $(0,w)$. 

The extremals of the families fill two manifolds $\mathcal{N}_w^{+}$ and $\mathcal{N}_w^{-}$,  each of them being of dimension $2$ and homeomorphic to $ \R^{2}$, coming respectively into and out of the point $(0,w)$. The switching points of $\mathcal{N}_w^{\pm}$ fill two piecewise-smooth curves $\Gamma_w^{\pm}$. 

The union $\cup_{w \in \mathcal{O}} \mathcal{N}_w^{\pm}$ of all those submanifolds is endowed with the bundle structure with base $\mathcal{O}$ and two-dimensional piecewise smooth fibers filled by chattering extremals.
\end{prop}

\begin{figure}[h]
\centering
      \includegraphics[scale = 0.7]{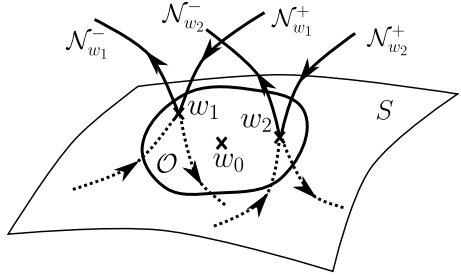}
        \caption{Phase portrait of optimal extremals near the singular surface.}
        \label{th1_portrait}
\end{figure}

Figure \ref{th1_portrait} illustrates Proposition \ref{thm1}. The extremals living in the submanifolds $\mathcal{N}_w^{+}$ and $\mathcal{N}_w^{-}$ are chattering. More precisely, the extremals in $\mathcal{N}_w^{+}$ reach $(0,w)$ (in finite time) with infinitely many switchings, and the extremals in $\mathcal{N}_w^{-}$ leave $(0,w)$ with infinitely many switchings. The submanifolds $\mathcal{N}_w^{\pm}$ can be seen as two-dimensional fibers.

\begin{proof}
The complete proof of Proposition \ref{thm1} is done in \cite{Zelikin1}. Let us however sketch the main steps. Assume that $z_2>0$.
\begin{enumerate}
\item Prove that there exist self-similar solutions (i.e., the one-parameter family of chattering solutions) for the unperturbed system \eqref{Hamsys} using the Poincar\'{e} mapping $\Phi$ of the switching surface to itself.
\item Prove that the points on $S$ are the stable points of $\Phi\circ\Phi$, by calculating the eigenvalues of $d(\Phi\circ\Phi)(0,w_0)$. Applying the invariant manifold theorem, there exists a one-dimensional $\Phi\circ\Phi$-invariant submanifold transversal to $S$ and passing through the point $(0,w_0)$. The restriction of $\Phi\circ\Phi$ to this submanifold is a contracting mapping. It follows the existence of a two-dimensional manifold $\mathcal{N}_{w_0}^+$ in the $(z,w)$-space, filled by chattering extremals entering into $(0,w_0)$. Moreover, the smooth dependence theorem leads to the bundle structure of $\cup_{w_0} \mathcal{N}_{w_0}^+$.
\item Prove that for the small perturbation system \eqref{small_perturbation}, the Poincar\'{e} mapping $\Phi$ is well defined and smooth at the points in the neighborhood of $\mathcal{N}_{w_0}^+$. Using similar techniques as in the first and second steps, prove that the solutions of the perturbed system have the same structure than that of the unperturbed system. 
\end{enumerate}
When $z_2<0$, another two-dimensional manifold $\mathcal{N}_{w_0}^-$ in $(z,w)$-space filled by chattering extremals that coming out of the point $(0,w_0)$ can be found and $\cup_{w_0} \mathcal{N}_{w_0}^-$ is also endowed with a bundle structure.
\end{proof}

The subbundles described in Proposition \ref{thm1} are given by
\begin{equation*}
	\Sigma^{\pm} = \cup_{w \in \mathcal{O}} \mathcal{N}_w^{\pm} ,
\end{equation*}
where the subbundle $\Sigma^{+}$ (resp., $\Sigma^{-}$) is filled by chattering arcs that come into (resp., come out of) the singular surface. Moreover, we denote the switching surfaces as ${\Gamma^{\pm} = \cup_{w \in \mathcal{O}} \Gamma_w^{\pm}}$.

Note that it suffices to consider only the subbundle $\Sigma^{+}$, since the properties of $\Sigma^{-}$ can be obtained similarly. We consider the canonical projection $\pi:\Sigma^{+} \to \mathcal{O}$ from the subbundle to the base.

\subsubsection{Optimality status}
We now raise the question of knowing whether these chattering extremals are optimal or not. Let us consider again the Fuller problem to give an intuitive idea. Using \eqref{XPtoZW}, we choose the new variables $z=(p_2, -p_1, -2 x_1, -2 x_2)$ and then clearly the singular surface coincides with the origin. According to Proposition \ref{thm1}, there are two integral submanifolds of dimension $2$ that are filled by chattering extremals coming into and out of the origin within finite time, with infinitely many switchings. 

We consider the canonical projection $\pi^\ast: (z,w) \rightarrow x$ from the $(z,w)$-space to the $x$-space (state space).
It is known that the extremals fill a Lagrangian submanifold in the $(z,w)$-space. Their projection on the state space are the trajectories, of which we would like to ensure their local optimality status. According to the conjugate point theory (see \cite{AGRACHEV,Bonnard}), it suffices to ensure that the projection $\pi^\ast$ be regular along the Lagrangian manifold (in other words, we require that its differential be surjective along that manifold).
Note that we can consider as well the projection from the $(x,p)$-space to the $x$-space, instead of $\pi^\ast$, because the coordinate change $(x,p) \mapsto (z,w)$ is bijective in the neighborhood of a point $(x,p) \in S$. Indeed, this coordinate only needs to be regular for providing the regularity of projection from $(x,p)$-space to $x$-space.

As illustrated on Figure \ref{fuller_ct}(a), the above regularity condition ensures that the trajectories in the $x$-space do not intersect each other before reaching the target point or submanifold, and thus ensures to avoid the loss of local optimality of the trajectories at the intersection point (i.e., the conjugate point). Figures \ref{fuller_ct}(b) and \ref{fuller_ct}(c) show the optimal synthesis of the chattering trajectories $\pi^\ast (\mathcal{N}_w^{+})$ and $\pi^\ast (\mathcal{N}_w^{-})$ for the Fuller problem, respectively. These chattering solutions do not intersect and they are locally optimal. 

\begin{figure}[h]
\centering
\includegraphics[scale = 0.85]{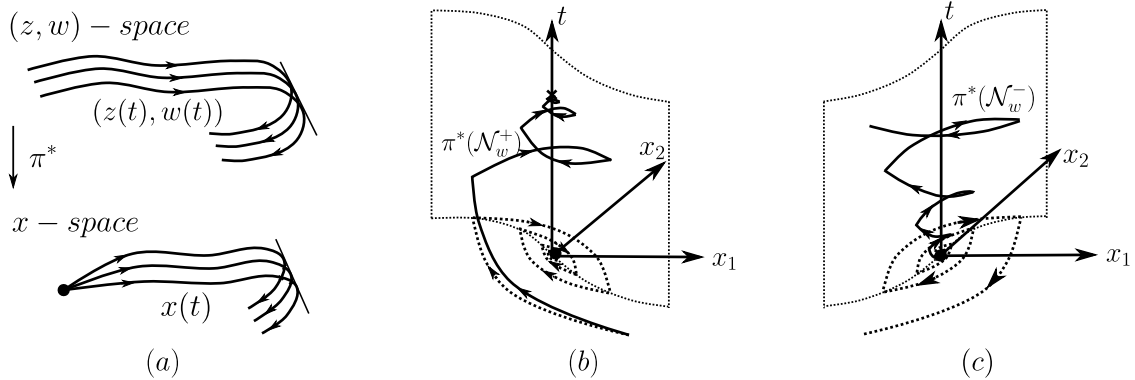}
\caption{ (a) Illustration of sufficient optimality condition; (b)-(c) Optimal synthesis of the Fuller problem.}
\label{fuller_ct}
\end{figure}

Let $M_1$ be a a target submanifold contained in the projection of the singular surface $\pi^\ast S$. For any point $x \in M_1$, we define its lift $(x,p(x))$ satisfying $(x,p(x))\in S$, $H(x,p(x))=0$ and $p(x)\,dx=0$ (transversality condition). The union $N$ of all such points $(x,p(x))$ must be transversal to the flow of the singular extremals in $S$. Thus, the singular extremals reaching the submanifold $N$ fill a submanifold $N^\ast$. In short, the submanifold $N$ is a lift of the target $M_1$ that intersects with the singular extremals. 

It is easy to see that the submanifold $N^\ast$ is Lagrangian. Hence the subbundle $\pi^{-1}(N^\ast)$ is Lagrangian as well. Therefore, according to the theory on Lagrangian manifolds and sufficient optimality conditions, it suffices to check the regularity of the projection $\pi^\ast$ restricted to $\pi^{-1}(N^\ast)$. 

The following proposition provides sufficient optimality conditions (see \cite{Zelikin2}) when the submanifolds $N$ and $N^\ast$ are of dimension $n-3$ and $n-2$ respectively.

\begin{prop}\label{prop_opti}
Consider the subbundle $\pi ^{-1} (N^{\ast})$ of the bundle $\Sigma^+$. Assume that the restriction of the projection $\pi^\ast$ on any smooth part of the bundle $\pi^{-1} (N^{\ast})$ is  regular and can be regularly extended to boundary points of the smooth part. Assume that the target manifold $M_1$ is connected. Then the projection of the solutions of the system \eqref{small_perturbation} filling $\pi ^{-1} (N^{\ast})$ are locally optimal in $C^0$ topology.
\end{prop}

The target submanifold has to be chosen adequately and must be of order $n-3$ in order to use this proposition. This condition on the dimension is used to take into account the two-dimensional fibers mentioned in Proposition \ref{thm1}.
\begin{figure}[h]
\centering
      \includegraphics[scale = 1.0]{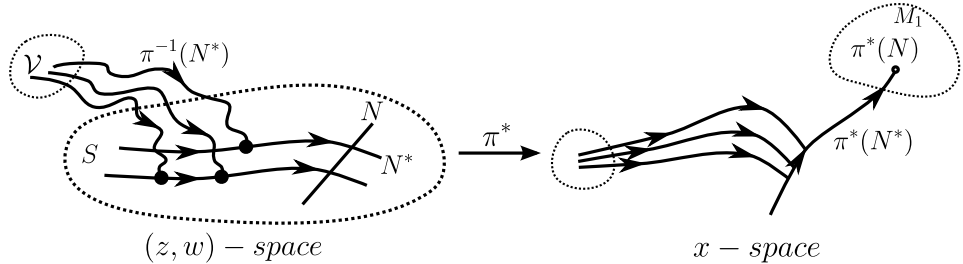}
        \caption{Illustration of Proposition \ref{prop_opti}.}
        \label{th2_illustration}
\end{figure}

As shown in Figure \ref{th2_illustration}, due to the endowed bundle structure, for every given initial point $(z_0,w_0)$ in the neighborhood of the singular surface $S$ in $(z,w)$-space, there is a neighborhood $\mathcal{V}$ of the point $(z_0,w_0)$ such that all extremals starting from the points inside $\mathcal{V}$ reach a point on $N^{\ast}$ in finite time with infinitely many switchings. Then, these extremals reach the target manifold $N$ along the singular extremals in $N^\ast$. If the projection $\pi^\ast$ is regular, then the projected trajectories in the $x$-space are locally optimal in $C^0$ topology.

\medskip

The condition of being a regular projection is the most difficult one to check. We set
\begin{equation*}
\Sigma^{\ast} = \pi ^{-1} (N^{\ast}), \quad
\Gamma^\ast = \Sigma^\ast \cap \Gamma^+, \quad
S_0 = S \cap \{H = 0\}.
\end{equation*}
In \cite{Zelikin2}, the authors provide the following sufficient condition for having a regular projection of $\Sigma^{\ast}$ into the $x$-space. 

\begin{lem} \label{lem_opti}
Let $\mathcal{L}$ be spanned by the vector $\partial / \partial z_3$ and by the vectors of the tangent plane to the switching surface $\Gamma^\ast$. 
Assume that the restriction of $d\pi^\ast$ to $\mathcal{L}$ is surjective. Then, the restriction of $\pi^\ast$ to $\Sigma^{\ast}$ is regular as well.
\end{lem}

\begin{rem} \label{rem_thm3}
Lemma \ref{lem_opti} indicates that $d\pi.\frac{\partial}{\partial z_3}$ should be transversal to the tangent plane to the switching surface of the chattering family generated by the submanifold $N$. Note that, at the points of the curve $N$, the tangent plane of the switching surface $\Gamma^\ast$ consists of three types of vectors: the nonsingular velocity vector, the singular velocity vector and the tangent vector to the curve $N$. 
\end{rem}

\section{Application to the planar tilting maneuver problem} 
In this section, we analyze the bang-bang, singular and chattering extremals of the problem $\MTTP$. We will see that, when the strategy involves a singular arc, then this singular arc is of intrinsic order two, and according to the previous section, this causes a chattering phenomenon. We will prove that chattering extremals are locally optimal in $C^0$ topology.

\subsection{Extremal equations}\label{sec_ReguArcs}
The Hamiltonian of the problem $\MTTP$ is of the form $H=h_0+ uh_1 + p^0$, where 
$h_0 = \langle p,f_0(x) \rangle$ and $h_1 = \langle p,f_1(x) \rangle=bp_4$, and the adjoint vector $p=(p_1,p_2,p_3,p_4)$ satisfies the adjoint equations
\begin{equation} \label{sys2}
\begin{cases}
\dot{p}_1 &= c(p_1 x_2 - 2 p_2 x_1 + p_3) ,\\
\dot{p}_2 &= c p_1 x_1,\\
\dot{p}_3 &= a (p_1 \sin x_3 - p_2 \cos x_3),\\
\dot{p}_4 &= - p_3.
\end{cases}
\end{equation}
Since $b>0$, we infer from the maximization condition of the PMP that $u(t) = \mathrm{sign} (p_4(t))$, provided that $\varphi(t)=bp_4(t)\neq 0$ (bang arcs).
The final condition $x(t_f)\in M_1$ yields the transversality condition
\begin{equation*} 
	p_1(t_f)\cos(\gamma_f)+p_2(t_f) \sin(\gamma_f) = 0.
\end{equation*}

\subsection{Lie bracket configuration of the system} \label{sec_Lie}
Before proceeding with the analysis of singular extremals, it is very useful to compute the Lie brackets of the vector fields $f_0$ and $f_1$ defined by \eqref{def_f0f1}. This is what we call the Lie bracket configuration of the control system \eqref{singleinputcontrolaffinesystem}.

\begin{lem}\label{Lieconfig}
We have
\begin{equation*}
\begin{split}
& f_0 = ( a \cos x_3  - c x_1 x_2 ) \frac{\partial}{\partial x_1} + (a \sin x_3 + c x_1^2- g_0 ) \frac{\partial}{\partial x_2} + (x_4 - c x_1) \frac{\partial}{\partial x_3} ,\\
& f_1 = b \frac{\partial}{\partial x_4} ,\quad [f_0,f_1] = -b \frac{\partial}{\partial x_3},\\
& [f_0,[f_0,f_1]] = -ab\sin x_3 \frac{\partial}{\partial x_1} + ab\cos x_3 \frac{\partial}{\partial x_2},\qquad [f_1,[f_0,f_1]] \equiv 0, \\
& \mathrm{ad}^3f_0.f_1 = -ab((x_4-2cx_1)\cos x_3+c x_2\sin x_3)\frac{\partial}{\partial x_1}-ab\sin x_3(x_4-3cx_1 )\frac{\partial}{\partial x_2} -abc\sin x_3 \frac{\partial}{\partial x_3}, \\
\end{split}
\end{equation*}
\begin{equation*}
\begin{split}
[f_1,[f_0,[f_0,f_1]] =& [f_0,[f_1,[f_0,f_1]] = [f_1,[f_1,[f_0,f_1]] =0 ,\\
\mathrm{ad}^4f_0.f_1 =& ab((-4cx_1x_4+cg_0 +x_4^2 +4c^2x_1^2-c^2x_2^2)\sin x_3-2ac+4ac \cos^2 x_3+(cx_1 \\
   & -2x_4)cx_2\cos x_3 )\frac{\partial}{\partial x_1} + ab(-c^2x_1x_2 \sin x_3+4ac \sin x_3\cos x_3+(-x_4^2+6cx_1x_4\\
   & -7c^2x_1^2)\cos x_3)\frac{\partial}{\partial x_2}+abc(3cx_1\cos x_3-2x_4\cos x_3-cx_2\sin x_3)\frac{\partial}{\partial x_3},\\
[f_1,\mathrm{ad}^3f_0.f_1] =& -ab^2\cos x_3 \frac{\partial}{\partial x_1} -ab^2\sin x_3 \frac{\partial}{\partial x_2}.\\
\end{split}
\end{equation*}
and
$$
\dim\mathrm{Span}(f_1,[f_0,f_1],[f_0,[f_0,[f_0,f_1]])=3
$$
\end{lem}
It follows from this lemma that the Poisson brackets $\{h_1,\{h_0,h_1\}\}$ and $\{h_1,\{h_0,\{h_0,h_1\}\}\}$ are identically equal to $0$. This is the main reason why we will have singular extremals of higher order, as shown in the next section.

\subsection{Singular extremals}\label{sec_SingArcs}
In this section, we compute all possible optimal singular extremals arcs. Later on, we are going to provide sufficient conditions on the initial conditions, under which the optimal strategy of the problem $\MTTP$ does not involve (optimal) singular arcs. Before that, let us first assume that singular arcs do exist, and let us establish some necessary conditions along them.

\begin{lem} \label{lem_singular}
Let $x(\cdot)$ be a singular arc, defined on the sub-interval $(t_1,t_2)$, and let $(x(\cdot),p(\cdot),p^0,u(\cdot))$ be an extremal lift. Then:
\begin{itemize}
\item along that singular extremal, we must have (omitting $t$ for readability)
\begin{equation} \label{switch_derives3}
\begin{split}
& p_1 \left( a -c x_1 x_2 \cos x_3  - (g_0 - c x_1^2) \sin x_3 \right) + p^0 \cos x_3 = 0 , \\
& p_2 \left( a -c x_1 x_2 \cos x_3  - (g_0 - c x_1^2) \sin x_3 \right) + p^0 \sin x_3 = 0, \\
& p_3 = p_4 = 0 ,
\end{split}
\end{equation} 
and
\begin{equation} \label{singularcontrol}
u  = \frac{c}{2b} \big( (-cx_2^2+2x_1x_4-3cx_1^2+g_0)\sin 2x_3+2cx_1x_2\cos 2x_3 +4a\cos x_3-4x_2x_4\cos^2 x_3 \big) ;
\end{equation}
\item $p^0 \neq  0$ (in other words, there is no abnormal singular extremal), and then we set $p^0=-1$;
\item the four constraints \eqref{switch_derives3} are functionally independent;
\item one has $\vert u(t)\vert < 1$, for almost every $t\in(t_1,t_2)$ (in other words, any singular arc is admissible);
\item $u$ is of intrinsic order two;
\item the strengthened generalized Legendre-Clebsch condition along the singular extremal reads
\begin{equation} \label{kelleyarea}
a -c x_1 x_2 \cos x_3  - (g_0 - c x_1^2) \sin x_3 > 0 .
\end{equation}
\end{itemize}
\end{lem}

In particular, the last item of the lemma states that optimal singular arcs, if they exist, must live in the region of the state space $\R^4$ defined by \eqref{kelleyarea}. The third item of the lemma implies that the singular extremals of the problem are in a submanifold of codimension $4$, i.e., the singular surface of $\MTTP$ is of codimension $4$.

\begin{proof}
Along the interval $I=(t_1,t_2)$ on which the singular arc is defined, the switching function $\varphi(t)=h_1(x(t),p(t))=bp_4(t)$ must be identically equal to zero. 
Differentiating with respect to time, we get that $\{ h_0,h_1 \} = - bp_3 = 0$ along $I$. 

Differentiating again, we get $\{h_0,\{ h_0,h_1 \}\} + u \{h_1,\{h_0,h_1\}\} = 0$, and since the Poisson bracket $\{h_1,\{h_0,h_1\}\}$ is identically equal to $0$ (see Lemma \ref{Lieconfig}), we have $\{h_0,\{ h_0,h_1 \}\} = \mathrm{ad}^2 h_0.h_1 =  -ab ( p_1 \sin x_3 -  p_2 \cos x_3) = 0$ along $I$ (and the equation $\{h_1,\{h_0,h_1\}\}=0$ does not bring any further information).

Differentiating again, we get $\{h_0,\{h_0,\{ h_0,h_1 \}\}\} + u \{h_1,\{h_0,\{h_0,h_1\}\}\} = 0$, and there, again from Lemma \ref{Lieconfig}, the Poisson bracket $\{h_1,\{h_0,\{h_0,h_1\}\}\}$ is identically equal to $0$ (and thus brings no additional information). Hence 
$$
\mathrm{ad}^3h_0.h_1 
=  -ab \big(x_4 (p_1 \cos x_3 + p_2 \sin x_3)+cp_1x_2\sin x_3 -3cp_2x_1\sin x_3-2cp_1x_1\cos x_3 \big)
= 0,
$$
which gives a new constraint.

Finally, a last derivation yields $\mathrm{ad}^4h_0.h_1 + u \{h_1,\mathrm{ad}^3h_0.h_1\} = 0$ and since $\{ h_1, \mathrm{ad}^3 h_0.h_1 \} \neq 0$, we infer that
$$
u = -\frac{\mathrm{ad}^4 h_0.h_1}{\{ h_1, \mathrm{ad}^3 h_0.h_1 \}},
$$
along $I$, and \eqref{singularcontrol} is obtained. Here, we have
$$
\mathrm{ad}^4 h_0.h_1 = \langle p, \mathrm{ad}^4f_0.f_1(x) \rangle,
$$
and
$$
\{ h_1, \mathrm{ad}^3 h_0.h_1 \}   = \langle p, [f_1, \mathrm{ad}^3f_0.f_1](x) \rangle= - a b^2 (p_1 \cos x_3 + p_2 \sin x_3) .
$$
Hence, we have obtained the constraints
\begin{equation*}
\begin{split}
& p_3 = p_4 = 0,\qquad p_1 \sin x_3 -  p_2 \cos x_3 = 0,\\
& -x_4(p_1\cos x_3+p_2\sin x_3)+3cp_2x_1\sin x_3+cp_1(2x_1\cos x_3-x_2\sin x_3)= 0 .
\end{split}
\end{equation*}
They are functionally independent because $\dim\mathrm{Span}(f_1,[f_0,f_1],[f_0,[f_0,[f_0,f_1]])=3$ (see Lemma \ref{Lieconfig}). 
Moreover, using the fact that $H \equiv 0$ along an extremal, we infer the relations \eqref{switch_derives3}. Setting
\begin{equation*}
\begin{split}
y_1 &= p_1 \left( a -c x_1 x_2 \cos x_3  - (g_0 - c x_1^2) \sin x_3 \right) + p^0 \cos x_3, \\
y_2 &= p_1 \left( a -c x_1 x_2 \cos x_3  - (g_0 - c x_1^2) \sin x_3 \right) + p^0 \sin x_3,
\end{split}
\end{equation*}
we have
$
\mathrm{rank} \frac{\partial (y_1,y_2,p_3,p_4)}{\partial (x,p)}=4,
$
provided that $p^0 \neq 0$ and $p_1 \neq 0$, $p_2 \neq 0$. This implies that these four functions are functionally independent. 
If $p_1=0$ or $p_2=0$, then it is easy to see that $p_1=p_2=p^0=0$, which violates the PMP. Hence $p_1 \neq 0$ and $p_2 \neq 0$.
If $p^0$ were to be zero, then it would follow from $p_1 \neq 0$ and $p_2 \neq 0$ that 
$y=a -c x_1 x_2 \cos x_3  - (g_0 - c x_1^2) \sin x_3 \equiv 0$ along $I$. Differentiating, we get $\dot{y} \equiv 0$ and $\ddot{y}=\alpha_c +u_c \beta_c \equiv 0$. By substituting $p_1 \sin x_3 =  p_2 \cos x_3$ into $-x_4(p_1\cos x_3+p_2\sin x_3)+3cp_2x_1\sin x_3+cp_1(2x_1\cos x_3-x_2\sin x_3=0$, we  get
$$
y_3= -x_4+cx_1(2+\sin x_3^2)-cx_2\sin x_3 \cos x_3 =0.
$$
Then, setting
$
y_4=u_c - u_s=-\alpha_c/\beta_c-u, 
$
we check that $y=0$, $\dot{y}=0$, $y_3=0$ and $y_4=0$ are four functionally independent constraints on the $x$-space. Hence, the trajectory along $I$ becomes some points. To stay along $I$ on this abnormal extremal, we need in addition $u=0$ which is another independent constraint, and so the number of constraints has exceeded the dimension of the extremal $(x,p)$-space. Therefore $p^0 \neq  0$.

Using the numerical values of Table \ref{sim_param}, we have
\begin{equation} \label{sc_real}
	|u| \leq \frac{c}{2b} \big( 4a+6v_m\omega_{max}+cv_m^2 \big) \leq 0.3  ,
\end{equation}
and thus $\vert u\vert < 1$. Hence, for the problem $\MTTP$, we have, along any singular extremal arc,
\begin{equation*}
\frac{\partial}{\partial u} \frac{d^k}{dt^k} h_1 = 0, \quad k=0,1,2,3  ,\qquad
\frac{\partial}{\partial u} \frac{d^4}{dt^4} h_1 = \beta(x,p)= -a b^2 (p_1 \cos x_3 + p_2 \sin x_3) \neq 0 ,
\end{equation*}
and then, according to Definition \ref{def_singularorder}, the singular solutions (which are admissible from \eqref{sc_real}) are of intrinsic order two. 
The strengthened generalized Legendre-Clebsch condition for the problem $\MTTP$ is written here as $\beta(x,p) <0$, and hence, using \eqref{switch_derives3} and taking $p^0 = -1$, we obtain \eqref{kelleyarea}.
\end{proof}

\begin{cor}
For the problem $\MTTP$, any optimal singular arc cannot be connected with a nontrivial bang arc. We must then have chattering, in the following sense. Let $u$ be an optimal control, solution of $\MTTP$, and assume that $u$ is singular on the sub-interval $(t_1,t_2)\subset[0,t_f]$ and is non-singular elsewhere. If $t_1>0$ (resp., if $t_2<t_f$) then, for every $\varepsilon>0$, the control $u$ switchings an infinite number of times over the time interval $[t_1-\varepsilon,t_1]$ (resp., on $[t_2,t_2+\varepsilon]$).
\end{cor}

\begin{proof}
This result follows from Lemma \ref{thm_bs} and Lemma \ref{lem_singular}. However, the proof is simple and we provide hereafterin the argument. 

It suffices to prove that the existence of an extremal consisting of the concatenation of a singular arc of higher order with a non-singular arc violates the PMP. The reasoning goes by contradiction.
Assume that $t_1>0$ and that there exists $\varepsilon>0$ such that $u(t) = 1$ over $(t_1-\epsilon,t_1)$. By continuity along the singular arc, we have $\varphi(t_1)=\varphi^{(1)}(t_1)=\varphi^{(2)}(t_1)=\varphi^{(3)}(t_1)=0$, and it follows from the strengthened generalized Legendre-Clebsch condition $\beta (x,p)< 0$ that
\begin{multline*}
	0=\varphi^{(4)}(t_1^+) = \mathrm{ad}^4h_0.h_1(t_1)+\{h_1,\mathrm{ad}^3h_0.h_1\}(t_1) u(t_1^+)	\\
	  > \mathrm{ad}^4h_0.h_1(t_1)+\{h_1,\mathrm{ad}^3h_0.h_1\}(t_1) u(t_1^-)=\varphi^{(4)}(t_1^-),
\end{multline*}
and hence the switching function $t\mapsto \varphi(t)=h_1(x(t),p(t))$ has a local maximum at $t=t_1$ and thus is nonnegative over $(t_1-\epsilon,t_1)$, provided that $\varepsilon>0$ is small enough. It follows from the maximization condition of the PMP that $u(t_1)=-1$ over $(t_1-\epsilon,t_1)$. This contradicts the assumption.
\end{proof}

\subsection{Optimality status of chattering extremals}\label{sec_ChatArcs}
In this section, we analyze the optimality status of chattering extremals in the problem $\MTTP$.

\begin{lem} \label{lem_chattering1}
Assume that $x_3 \neq  \pi/2 + k\pi$, $k \in \mathbb{Z}$. The Hamiltonian system, consisting of \eqref{sys1} and \eqref{sys2}, can be written as a small perturbation system, in the form \eqref{Hamsys}, as
\begin{equation} \label{as_S}
\begin{cases}
\dot{z_1} &= z_2 ,\\
\dot{z_2} &= z_3 ,\\
\dot{z_3} &= z_4 + f_3(z,w,u),\\
\dot{z_4} &= \alpha_0(w) + u \beta_0(w) + f_4(z,w,u),\\
\dot{w} &= F(z,w,u) .\\
\end{cases}
\end{equation}
where $u \in [-1,1]$ and
\begin{equation} \label{smallpertcond}
	\lim_{\lambda \to +0} \frac{f_3(G_{\lambda}(z),w,u)}{\lambda^{(5-3)}}=0, \:\:
	\lim_{\lambda \to +0} \frac{f_4(G_{\lambda}(z),w,u)}{\lambda^{(5-4)}} < \infty .
\end{equation}
by choosing new variable $(z,w)$ as
\begin{equation} \label{var_change}
\begin{cases}
	z_1 = p_4,\:\: z_2 = p_4 ^{(1)},\:\: z_3 = p_4 ^{(2)},
	z_4 = p_4 ^{(3)} + a c \sin x_3 p_3,\\
	\d{w_1 = a ( p_1 \sin x_3 +  p_2 \cos x_3)} , \\
	\d{w_2 = a(x_4\cos x_3+cx_2\sin x_3-2cx_1\cos x_3)p_1}+a\sin x_3 (-x_4+3cx_1)p_2,\\
	\d{w_3 = p_2/p_1} ,\\
	\d{w_4 = x_1} .
\end{cases}
\end{equation}
in the neighborhood of the singular surface defined by $z=0$ in the $(z,w)$-space. In addition, the strengthened generalized Legendre-Clebsch condition for system (\ref{as_S}) yields 
\begin{equation} \label{kelleyarea2}
	w_1 w_3>0 .
\end{equation}
\end{lem}

\begin{proof}
We have proved that $p_1 \neq  0$ and $p_2 \neq  0$ along the singular arc. Then, from $x_3 \neq  \pi/2 + k\pi$, $k \in \mathbb{Z}$  we can prove that the Jacobi matrix of this variable change is of full rank by direct calculations, i.e.,
\begin{equation*} 
	\mathrm{rank} \big( \frac{D(z,w)}{D(x,p)} \big)=  8,
\end{equation*}
After some manipulations, we can express $(x,p)$ by the new variables $(z,w)$ chosen in (\ref{var_change}), as
\begin{equation}\label{Var_Trans}
\begin{split} 
	&\d{x_1=w_4}, \quad \d{x_4=3c w_4 - \frac{z_4+w_2}{w_3(w_1-z_3)}}, \quad p_3=-z_2, \quad p_4=z_1,\\
	&\d{x_2=\frac{w_1w_2-w_1z_4-w_2z_3+z_3z_4}{c(w_1-z_3)^2}+\frac{w_1w_2+w_1z_4+w_2z_3+z_3z_4-cw_3(w_4w_1^2-w_4z_3^2)}{cw_3^2(w_1-z_3)^2}} ,\\
	&x_3 = -2 \mathrm{arctan}\big(w_1+z_3 \pm(w_1^2w_3^2+w_1^2-2w_1w_3^2z_3+2w_1z_3+w_3^2z_3^2+z_3^2)/(w_3w_1-w_3z_3) \big),\\
	&\d{p_1=\mp \sqrt{w_1^2w_3^2+w_1^2-2w_1w_3^2z_3+2w_1z_3+w_3^2z_3^2+z_3^2}/ (2aw_3)},\\
	&\d{p_2=\mp \sqrt{w_1^2w_3^2+w_1^2-2w_1w_3^2z_3+2w_1z_3+w_3^2z_3^2+z_3^2}/ (2a)}.
\end{split}
\end{equation}
Although this variable transformation is not one to one in the whole $(x,p)$-space, it does not matter, because the semi-canonical system we use is a local system and so we just need to consider separately the domain $x_3 \in \mathcal{D}_1=(-\pi/2,\pi/2)$ and $x_3 \in \mathcal{D}_2=(-\pi,-\pi/2)\cup(\pi/2,\pi)$. 

The manifold of singular trajectories specified by $z=0$ can be written as
\begin{equation*}
	S = \{(x,p)|p_3=0,p_4=0,p_2=p_1\tan x_3,x_4=cx_1(2+\sin^2 x_3)-cx_2\sin x_3 \cos x_3 \} .
\end{equation*} 
Differentiating $(z,w)$ defined in \eqref{var_change} with respect to time with the help of \eqref{sys1} and \eqref{sys2}, we get the system in form,
\begin{equation} \label{syssemixp}
\begin{cases}
\dot{z_1} &= z_2,\\
\dot{z_2} &= z_3,\\
\dot{z_3} &= z_4+f_3(x,p),\\
\dot{z_4} &= A(x,p)+B(x,p) u,\\
\dot{w}   &= F(x,p,u),
\end{cases}
\end{equation}
with\begin{equation*}
	u=\begin{cases}
		1 & \textrm{if}\ z_1 >0 ,\\
		-1 &\textrm{if}\  z_1 <0 ,\\
		-A(x,p)/B(x,p) &\textrm{if}\  z_1 =0 ,
	\end{cases}
\end{equation*}
where
\begin{equation*}
	A(x,p) =  \{ h_0(x,p), z_4 \} ,\quad
	B(x,p) =  \{ h_1(x,p), z_4 \} 
	       = \beta(x,p)/b .
\end{equation*}
Note that here we have $u = -A/B = u_s$, where $u_s$ is given in \eqref{singularcontrol}. Hence we infer $|u|<1$. By substituting \eqref{Var_Trans} into \eqref{syssemixp}, we can obtain $f_3(z,w)$, $A(z,w)$, $B(z,w)$ and $F(z,w,u)$. 
Then we expand $A(z,w)$ and $B(z,w)$ in the vicinity of $S$ by
\begin{equation*}
	A(z,w)=A(0,w)+ \sum_{k=1}^{\infty} \frac{\partial^k A}{\partial z^k}(0,w) \frac{z^k}{k !},\:\:
	B(z,w)=B(0,w)+ \sum_{k=1}^{\infty} \frac{\partial^k B}{\partial z^k}(0,w) \frac{z^k}{k !}.
\end{equation*}
By taking $\alpha_0(w) = A(0,w)$, $\beta_0(w) = B(0,w)$, system (\ref{as_S}) is derived. We can see that system (\ref{as_S}) is a small perturbation of system (\ref{Hamsys}) since condition (\ref{sp_cond}) holds, i.e., condition \eqref{smallpertcond} holds.
Moreover, the strengthened generalized Legendre-Clebsch condition \eqref{kelleyarea2} is derived from 
\begin{equation*}
   \beta_0(w)=-b\frac{w_1(1+w_3^2)}{2w_3} < 0 ,
\end{equation*}
and \eqref{syssemixp} is transformed into a small perturbation system of form \eqref{as_S}.
\end{proof}

The functions $f_3(z,w)$, $\alpha_0(w)$ and $F(z,w,u)$ are different for $x_3 \in \mathcal{D}_1$ and $x_3 \in \mathcal{D}_2$. However, we will see next that this difference does not have any influence in the demonstration of the optimality result of chattering extremals.

\begin{cor} \label{cor_chattering1}
For the problem $\MTTP$, there exit two subbundles $\Sigma^{+}$ and $\Sigma^{-}$ having the singular surface $S$ as a base, and two fibers $\mathcal{N}^+$ and $\mathcal{N}^-$ of dimension two filled by chattering solutions.
\end{cor}

\begin{proof}
It suffices to apply Lemma \ref{lem_chattering1} and Proposition \ref{thm1}.
\end{proof}

We define $S_0 = S \cap \{ H\equiv 0 \}$. Let us consider an optimal solution $x(\cdot)$ of $\MTTP$, and let us assume that $x(\cdot)$ contains a singular arc defined on $(t_1,t_2)$. Let
\begin{equation*} 
	M_1^\ast = \{ x_2 = \Psi_1(x_1) \} \cap \pi^\ast(S_0) , 
\end{equation*}
be the submanifold where the extremals come into and out of the image of the singular surface $\pi^\ast(S_0)$, as shown in Figure \ref{targets}. 

\begin{figure}[h]
\centering
	 \includegraphics[scale = 1]{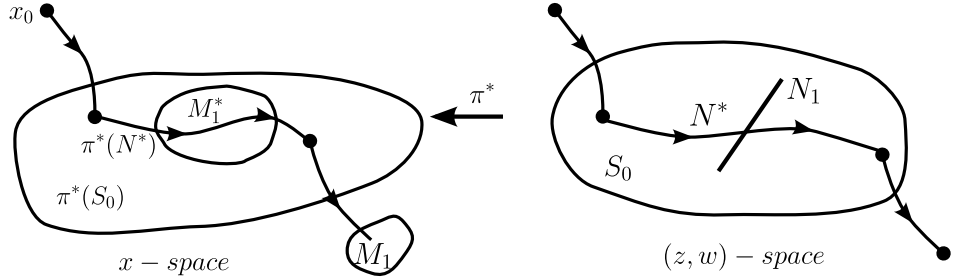}
	\caption{Illustration of $M^\ast_1$.}
	\label{targets}
\end{figure}

In the sequel, we want to analyze the optimality status of the chattering solutions with the ``target'' submanifold $M^\ast_1$. The optimality status of the chattering solutions starting from the submanifold $M^\ast_1$ can be analyzed similarly by considering the subbundle $\Sigma^-$.

We denote by $N_1$ the lift of $M^\ast_1$ in $(x,p)$-space by associating $x \in M^\ast_1$ with the point $(x,p(x))$ that belongs to $S_0$ and satisfies the transversality condition $p_1 = -p_2 \Psi_1^{\prime}(x_1) $ (following from \eqref{Hamiltonsys3}).

\begin{lem} \label{lem_taget}
The submanifold $N_1$ is Lagrangian submanifold of $\R^8$ of codimension 7. Moreover, the function $\Psi_1(\cdot)$ can be chosen such that the submanifold $N_1$ is transversal to the velocity vector of the singular extremals in $S$.
\end{lem}

\begin{proof}
From the definition of $N_1$ and \ref{lem_singular}, we infer that that $(x,p(x))$ satisfies
\begin{equation*} 
\begin{split}
& p_1 = \frac{\cos x_3}{a -c x_1 x_2 \cos x_3  - (g_0 + c x_1^2) \sin x_3}, \\
& p_2 = \frac{\sin x_3}{a -c x_1 x_2 \cos x_3  - (g_0 + c x_1^2) \sin x_3} , \\
& p_3 = 0, \:\: p_4 = 0,\\
& x_2 - \Psi_1(x_1) = 0,\\
& - x_4 + cx_1(1+\sin^2 x_3) - cx_2 \sin x_3 \cos x_3=0,\\
& \Psi_1^{\prime}(x_1) \tan x_3+1=0.	
\end{split}
\end{equation*}
Then, the $x$-component of the tangent vector to $N_1$ can be written as
\begin{equation*} 
	v_1 = \left( 1,\frac{\partial x_2}{\partial x_1},\frac{\partial x_3}{\partial x_1},\frac{\partial x_4}{\partial x_1}\right)^\top ,
\end{equation*}
where
\begin{equation*} 
\begin{split}
	\frac{\partial x_2}{\partial x_1} &= \Psi_1^{\prime}(x_1),\quad
	\frac{\partial x_3}{\partial x_1} = -\frac{\Psi_1^{\prime \prime}(x_1)}
	                                       {\Psi_1^{\prime}(x_1)} \sin x_3 \cos x_3,\\
	\frac{\partial x_4}{\partial x_1} &= c(2+\sin^2 x_3)
	                                   +\frac{c}{4}\frac{\Psi_1^{\prime \prime}(x_1)}{\Psi_1^{\prime}(x_1)}
	                                   \big( 2x_1(2+\sin 2x_3) \sin 2x_3-x_2 \sin 4x_3 \big).
\end{split}
\end{equation*}
Therefore, the $1$-form $\bar{\omega} = p dx = p_1 dx_1 + p_2 dx_2 + p_3 dx_3 + p_4 dx_4$ vanishes on every tangent vector to the submanifold $N_1$. Thus $N_1$ is of codimension $7$ and it is Lagrangian.

Moreover, the $x$-component of the velocity on the singular trajectories is
\begin{equation*} 
	v_2 = (\dot{x}_1,\dot{x}_2,\dot{x}_3,\dot{x}_4)=( a \cos x_3  - c x_1 x_2 , a \sin x_3 + c x_1^2- g_0, x_4 - c x_1, bu_s)^\top ,
\end{equation*}
with $u = - a(w)/b(w)$. Hence, to provide transversality, it suffices to choose the function $\Psi_1$ such that $v_1$ and $v_2$ are not proportional, e.g., $\Psi_1^{\prime} \neq  \frac{a \sin x_3 + c x_1^2-g_0}{a \cos x_3 - c x_1 x_2}$.
\end{proof}

It follows from this lemma that the submanifold $N^\ast$ filled by singular extremals coming into $N_1$ is Lagrangian. According to Proposition \ref{prop_opti}, it suffices to prove the regularity of the projection $\pi^\ast$ on $\pi^{-1}(N^\ast)$ using \ref{lem_opti}.

We denote by $v_3$ the nonsingular velocity vector and by $v_k$ the derivative of the projection $\pi^\ast$ of $\d{\partial / \partial z_3}$. We set $V = (v_1, v_2, v_3, v_k)$.

\begin{thm}\label{thm24}
If the function $\Psi_1(\cdot)$ is chosen such that	
\begin{equation} \label{cond_rank}
	\det V \neq  0,
\end{equation}
then the chattering solutions of the problem $\MTTP$ are locally optimal in $C^0$ topology.
\end{thm}

\begin{proof}
On $S_0$ we have
\begin{equation*} 
	\d{d\pi^\ast\left(\frac{\partial}{\partial z_3}\right) = 
	      \frac{\partial x_1}{\partial z_3} \frac{\partial}{\partial x_1} 
		+\frac{\partial x_2}{\partial z_3} \frac{\partial}{\partial x_2} 
		+\frac{\partial x_3}{\partial z_3} \frac{\partial}{\partial x_3} 
		+\frac{\partial x_4}{\partial z_3} \frac{\partial}{\partial x_4}} ,
\end{equation*}
where
\begin{align*}
	&\d{\frac{\partial x_1}{\partial z_3}=0},
	\quad \d{\frac{\partial x_2}{\partial z_3}=\frac{w_2 w_3^2-2cw_1w_4w_3+3w_2}{cw_1^2w_3^2}}, 
	\quad \d{\frac{\partial x_3}{\partial z_3}=-\frac{2w_3}{w_1(1+w_3^2)}}, 
	\quad \d{\frac{\partial x_4}{\partial z_3}=-\frac{w_2}{w_1^2w_3}} , \\
\end{align*}
and hence it follows that 
\begin{equation*} 
	d\pi^\ast\left(\frac{\partial}{\partial z_3}\right) = \left(0,\frac{w_2 w_3^2-2cw_1w_4w_3+3w_2}{cw_1^2w_3^2},-\frac{2w_3}{w_1(1+w_3^2)},-\frac{w_2}{w_1^2w_3}\right)^\top.
\end{equation*}
Denote $\d{d\pi^\ast\left(\frac{\partial}{\partial z_3}\right)}$ as $v_k$. Using \eqref{switch_derives3} and \eqref{var_change} we can get $v_k(w)$ as a vector depending on state variable $x$, i.e. $v_k(x)$. 

According to Lemma \ref{lem_opti} and Remark \ref{rem_thm3}, if $v_1$, $v_2$, $v_3$ and $\d{d\pi\left(\frac{\partial}{\partial z_3}\right)}$ are linearly independent, then the projection is regular. In our problem, we have that the nonsingular velocity vector associated with $u = 1$ is
\begin{equation*} 
	v_3 = (\dot{x}_1,\dot{x}_2,\dot{x}_3,\dot{x}_4)=( a \cos x_3  - c x_1 x_2 , a \sin x_3 + c x_1^2- g_0, x_4 - c x_1, b)^\top .
\end{equation*}

Therefore, if the condition \eqref{cond_rank} is satisfied on the points of $N_1$, then, the curve $N_1$ has been chosen such that the conditions of Proposition \ref{prop_opti} are fulfilled and so it generates the field of locally optimal chattering solutions in $C^0$ topology.
\end{proof}

\begin{rem}
The condition \eqref{cond_rank} in Theorem \ref{thm24} is always satisfied if one chooses an appropriate function $\Psi_1(\cdot)$. Indeed, we have
$$
\det V = \sum_{i=1}^4 v_{i,1} D_{i,1},
$$
where $D_{i,1}$ is the $(i,1)$ minor. Some calculations show that $D_{4,1}=0$, and hence \eqref{cond_rank} becomes $\det V = D_{1,1}-\Psi_1' D_{2,1}-\Psi_1''/\Psi_1' D_{3,1} \neq 0$. It suffices to ensure that $D_{i,1}$, $i=1,2,3$, do not vanish simultaneously. We prove this fact by contradiction: otherwise, it is easy to check that they yield three independent constraints in the $x$-space, and moreover, they are independent of the constraints $y_3=0$ and $x_2=\Psi(x_1)$ for the singular surface $S$. In this case, the number of constraints is larger than the dimension of the $x$-space. Therefore, $D_{i,1}$, $i=1,2,3$ do not vanish simultaneously. 
\end{rem}

\section{Chattering prediction}\label{sec_SingPred}
Since the chattering phenomenon causes deep difficulties for practical implementation, due to the fact that an infinite number of control switchings within finite time cannot be realized in real-life control strategies, in this section our objective is to provide precise conditions under which we can predict that an optimal singular arc does not appear, and thus there is no chattering arcs. 

A maneuver with $\gamma_f \geq \gamma_0$ (resp., $\gamma_f < \gamma_0$) is said to be a \textit{anticlockwise} maneuver (resp., a \textit{clockwise} maneuver).
In practice, the values of $x_{30}$ and $x_{3f}$ are chosen in $(0,\pi/2)$, and the values of $\gamma_0$, $x_{40}$ are chosen such that $| \gamma_0-x_{30} | \leq 0.1$ and $|x_{40}| \leq 0.1$.

We distinguish between those two maneuvers because of the gravity force imposed to the spacecraft. We will see further, in the numerical results, that the clockwise maneuver is easier to perform than the anticlockwise maneuver, in the sense that the clockwise maneuver time is shorter (in time), and there is less possibility of encountering a singular arc. This fact is due to the nonlinear effects caused by the gravity force. Indeed, intuitively, the gravity force tends to reduce the value of $x_2$, and hence the value of $\tan \gamma=x_2/x_1$ tends to get smaller. Then the spacecraft velocity turns naturally to the ground (pitch down) under the effect of the gravity. This tendency helps the clockwise maneuver to be ``easier''.

Although we have set $c=10^{-6}$ (see Table \ref{sim_param}), we have $c \leq 10^{-6}$ in real-life, since $c=1/r$ where $r$ is a distance not less than the radius of the Earth. The case  $c=0$ corresponds to a \textit{flat-Earth case}, and the case $c \in (0,10^{-6}]$ will be referred to as the \textit{non-flat case}.

\subsection{Flat-Earth case $c = 0$} \label{sec_limicase}
We can see from \eqref{flightangle} that $\dot{\gamma}$ is much smaller than $\dot{x}_3$ and $\dot{x}_4$ given by \eqref{singleinputcontrolaffinesystem}. Therefore, the main factor that affects the total maneuver time is the time to change $\gamma$ from $\gamma(0) = \gamma_0$ to $\gamma(t_f) = \gamma_f= x_{3f}$. In order to shorten the maneuver time, it is required to keep $\dot{\gamma}$ as large as possible. 

To this aim, we consider the time minimum control problem in which $x_3$ is seen as a control. We call this problem the \emph{problem of order zero}, i.e.,
\begin{align*}
& \min t_f \quad s.t.\\
& \dot{x}_1= a \cos x_3, \quad \dot{x}_2= a \sin x_3-g_0,\\
& x_1(0)=v_0 \cos \gamma_0,\quad x_2(0)= v_0 \sin \gamma_0, \quad x_2(t_f)-x_1(t_f) \tan \gamma_f = 0,
\end{align*}
The optimal solution of this problem is easy to compute (explicitly) with the PMP. The optimal control on $[0,t_f]$ is given by
\begin{equation} \label{x3_s}
	x_3(t) = x_3^{\ast} =\begin{cases}
	\gamma_f + \pi/2, & \textrm{if}\ \ x_{20}\cos \gamma_f  \leq x_{10}\sin \gamma_f,\\
	\gamma_f -  \pi/2, & \textrm{if}\ \ x_{20}\cos \gamma_f  > x_{10}\sin \gamma_f,
	\end{cases}
\end{equation}
and the maneuver time is 
$$
t_f =\frac{(x_{10}\tan \gamma_f-x_{20})\cos \gamma_f}{a \sin(x_3^\ast-\gamma_f)-g_0\cos \gamma_f}.
$$
Moreover, the adjoint vector is given by
$$
(p_1,p_2)=(\sin \gamma_f,-\cos \gamma_f) \frac{p^0}{a \sin (x_3^\ast-\gamma_f)-g_0\cos \gamma_f}.
$$

\begin{rem}
From this expression, we see that $g_0$ makes the anticlockwise maneuver slower, i.e.,
$
t_f \geq \frac{(x_{10}\tan \gamma_f-x_{20})\cos \gamma_f}{a \sin(x_3^\ast-\gamma_f)},
$
and the clockwise maneuver faster, i.e.,
$
t_f \leq \frac{(x_{10}\tan \gamma_f-x_{20})\cos \gamma_f}{a \sin(x_3^\ast-\gamma_f)}.
$
Therefore, the clockwise maneuver is ``easier'' to perform than the anticlockwise maneuver thanks to the gravity, which corresponds to intuition, as saif at the beginning of this section.
\end{rem}

Turning back to the problem $\MTTP$ in the flat-Earth case, from Lemma \ref{lem_singular}, the singular surface is given by
\begin{multline*}
 S=\{(x,p) \mid x_3=x_3^\ast,\quad x_4=0, \quad p_1=-p^0\cos x_3^\ast /(a-g_0\sin x_3^\ast),\\
 p_2 =-p^0\sin x_3^\ast /(a-g_0\sin x_3^\ast),\quad p_3=0, \quad p_4=0 \}.
\end{multline*}
It is interesting to see that the solution of the problem of order zero coincides with the singular solution of problem $\MTTP$ in the flat-Earth case. We have the following results.

\begin{lem}
Let $x(\cdot)$ be an optimal solution of $\MTTP$ in the flat-Earth case, associated with the control $u$. If $x(\cdot)$ contains at most one point of $S_2=\{ (x,p)| x_3=x_3^\ast\}$, then the control $u$ is bang-bang and switches at most two times.
\end{lem}

\begin{proof}
If $u(\cdot)$ is singular, then $(x(\cdot),p(\cdot))$ is contained in $S \subset S_2$. From the definition of $S$, it is easy to prove that $x(t_1) \neq x(t_2)$ for any $t_1 \neq t_2$ in $[0,t_f]$, which means that $(x(t_1),p(t_1))$ and $(x(t_2),p(t_2))$ are two different points of $S_2$. This contradicts the condition that $x(\cdot)$ contains at most one point of $S_2$. Therefore, $u(\cdot)$ is bang-bang.

It suffices to prove that if $x(\cdot)$ contains at most one point of $S_2$, then $\ddot{\varphi}(t)$, $t \in [0,t_f]$, remains of constant sign and has at most one zero. Indeed, if this is true, then $\dot{\varphi}(t)=-p_3(t)$ is monotone, and it follows that the first derivative of the switching function $\varphi(t)$ has at most two zeros, which means that the control $u$ has at most two switchings. Let us prove this fact by contradiction.
If there exists $t_1 \in [0,t_f]$ such that $(x(t_1),p(t_1)) \in S_2$, using $p_1+\tan \gamma_f p_2=0$ (transversality condition) and $p_1,p_2 \neq 0$, we have 
$$
\ddot{\varphi}(t_1) = -\dot{p}_3(t_1)=-a(p_1\sin x_3-p_2\cos x_3) = a p_2\cos(x_3-\gamma_f)/\cos \gamma_f=0.
$$
From the continuity of $\ddot{\varphi}(\cdot)$, we get that there exist two times $\tau_i<t_1$ and $\tau_j>t_1$ such that $\ddot{\varphi}(\tau_i)\ddot{\varphi}(\tau_j)<0$. It follows that $x_3(t_1)$ and $x_3(t_2)$ are on different sides of $x_3^\ast=\gamma_f \pm \pi/2$, i.e.,
$
(x_3(t_1)-x_3^\ast)(x_3(t_2)-x_3^\ast)<0.
$
However, we know that $x_{30}$ and $x_{3f}=\gamma_f$ are on the same side of $x_3^\ast$, i.e.,
$
(x_{30}-x_3^\ast)(x_{3f}-x_3^\ast)<0,
$
and hence there must exist another time $t_2$ at which $\ddot{\varphi}(t_2)=0$ ($(x(t_2),p(t_2)) \in S_2$) in order to allow the trajectory to reach the terminal submanifold. This is a contradiction.
\end{proof}

We denote a bang arc with $u=1$ (resp. $u=-1$) as $A_+$ (resp. $A_-$), and we denote a chattering arc and a singular arc by $A_c$ and $A_s$, respectively.
Let $\mathcal{F}_{x_3}$ be the union of all trajectories $x(\cdot)$ consisting of three different bang arcs satisfying the terminal conditions $x(0)=x_0$ and $x_3(t_f)=x_{3f}$, $x_4(t_f)=0$. These trajectories are of the form $A_+A_-A_+$ or $A_-A_+A_-$.
Easy calculations show that the optimal control $u(t)$ and the trajectory $x(t)$ of the form $A_+A_-A_+$ (resp. $A_-A_+A_-$) are given by
\begin{equation*} 
u(t) = \begin{cases}
 +1, t\in [0,\tau_1),\quad (\textrm{resp.}, -1)\\
 -1, t\in [\tau_1,\tau_2)\cup[\tau_2,\tau_3), \quad (\textrm{resp.}, +1)\\
 +1, t\in [\tau_3,t_f], \quad (\textrm{resp.}, -1)
\end{cases}
\end{equation*}
and
\begin{equation} \label{pitchup_x3}
\begin{split}
x_1(t) & = v_0 \cos \gamma_0 +\int_0^t a\cos x_3(s) ds,\\
x_2(t) & = v_0 \sin \gamma_0 +\int_0^t a\sin x_3(s)-g_0 ds,\\
x_3(t)& = \begin{cases}
x_{30} + x_{40} t + b t^2/2,\quad t \in [0,\tau_1),\quad 
(\textrm{resp.}, x_{30} - x_{40} t- b t^2/2,)\\
x_3(\tau_1) + (x_{40} + b\tau_1) (t - \tau_1) - b(t - \tau_1)^2/2,\quad t \in [\tau_1,\tau_2),\\
(\textrm{resp.}, x_3(\tau_1) - (x_{40} + b\tau_1) (t - \tau_1) + b(t - \tau_1)^2/2,)\\
\bar{x}_3 - b(t-\tau_2)^2/2,\quad t \in [\tau_2,\tau_3),\quad 
(\textrm{resp.}, \bar{x}_3 + b(t-\tau_2)^2/2,) \\
x_3(\tau_3) - b(\tau_3 - \tau_2)(t - \tau_3) + b(t - \tau_3)^2/2,\quad t \in [\tau_3,t_f],\\
(\textrm{resp.}, x_3(\tau_3)+b(\tau_3 - \tau_2)(t - \tau_3)-b(t - \tau_3)^2/2,)
\end{cases} \\
x_4(t)& = \begin{cases}
x_{40}+bt,\quad t \in [0,\tau_1), (\textrm{resp.}, x_{40}-bt,)\\
x_{40}+b\tau_1-b(t-\tau_1),\quad t \in [\tau_1,\tau_2), (\textrm{resp.}, x_{40}-b\tau_1+b(t-\tau_1),)\\
-b(t-\tau_2),\quad t \in [\tau_2,\tau_3), (\textrm{resp.}\:b(t-\tau_2),)\\
-b(\tau_3-\tau_2)+b(t-\tau_3),\quad t \in [\tau_3,t_f],(\textrm{resp.}, b(\tau_3-\tau_2)-b(t-\tau_3),)
\end{cases}
\end{split}
\end{equation}
with
\begin{equation*}
\tau_1 = - \frac{x_{40}}{b} + \sqrt{\frac{x_{40}^2}{2b^2} - \frac{x_{30} -\bar{x}_3}{b}}, \quad \tau_2 = 2 \tau_1 + \frac{x_{40}}{b}, \quad
\tau_3 = \tau_2 + \sqrt{- \frac{x_{3f} - \bar{x}_3}{b}}, \quad t_f = 2 \tau_3  - \tau_2,
\end{equation*}
\begin{multline*}
\Big( \textrm{resp.},
\tau_1 = - \frac{x_{40}}{b} + \sqrt{\frac{x_{40}^2}{2b^2} + \frac{x_{30} - \bar{x}_3}{b}},\quad \tau_2 = 2 \tau_1 + \frac{x_{40}}{b}, \quad
\tau_3 = \tau_2 + \sqrt{  \frac{x_{3f} - \bar{x}_3}{b}}, \\
 t_f= 2 \tau_3  - \tau_2, \Big)
\end{multline*}
where $\bar{x}_3$ is the maximal (resp., minimal) value of $x_3(t)$, $t\in[0,t_f]$. Besides, by integration, we have
\begin{equation*}
p_3(t)=p_3(0)+\int_0^t a(p_1\sin x_3(\tau)-p_2\cos x_3(\tau)) \, d\tau = p_3(0)-\frac{ap_2}{\cos \gamma_f}\int_0^t \cos(x_3(\tau)-\gamma_f)\, d\tau,
\end{equation*}
and
\begin{equation} \label{p4t}
p_4(t)=p_4(0)-p_3(0)t+\frac{ap_2}{\cos \gamma_f}\int_0^t \int_0^s \cos(x_3(\tau)-\gamma_f) \, d\tau \,ds.
\end{equation}
Using $p_4(\tau_1)=p_4(\tau_3)=0$, we get
\begin{equation} \label{p30}
p_3(0)=\frac{ap_2}{(\tau_3-\tau_1) \cos \gamma_f} \left( \int_0^{\tau_3} \int_0^s \cos(x_3(\tau)-\gamma_f) \,d\tau\,ds -\int_0^{\tau_1} \int_0^t \cos(x_3(\tau)-\gamma_f) \,d\tau\,ds \right),
\end{equation}
and
\begin{multline} \label{p40}
p_4(0)=\frac{ap_2}{\cos \gamma_f} 
\Big(\frac{ \tau_1}{(\tau_3-\tau_1)} \int_0^{\tau_3} \int_0^s \cos(x_3(\tau)-\gamma_f) \,d\tau\,ds 
\\
-\frac{ \tau_3}{(\tau_3-\tau_1)} \int_0^{\tau_1} \int_0^s \cos(x_3(\tau)-\gamma_f) \,d\tau\,ds \Big).
\end{multline}
From $H(0)=0$ and using the transversality condition, we infer $p_1$ and $p_2$ as functions of $\bar{x}_3$ provided $p^0 \neq 0$. Actually, $p^0$ is indeed nonzero, otherwise, using  $H(0)=0$, the transversality condition and equations \eqref{p30} and \eqref{p40}, we would infer that $p=0$, which is absurd.
We see that, if moreover $x_2(t_f)=x_1(t_f)\tan x_{3f}$, then the trajectories $x(t)$ in $\mathcal{F}_{x_3}$ together with $p(t)$ satisfy all necessary conditions of the PMP.

The value $\bar{x}_3$ can be numerically derived from the condition $x_2(t_f)=x_1(t_f)\tan x_{3f}$, and then $(x(t),p(t))$ is obtained. In fact, for given terminal conditions $x(0)=x_0$, $x_3(t_f)=x_{3f}$ and $x_4(t_f)=0$, $\mathcal{F}_{x_3}$ can be seen as a one-parameter family of trajectories with parameter $\bar{x}_3$.
Hence, for any given $\d{ \bar{x}_3 \in (\max(\bar{x}_{30},x_{3f}),x_3^\ast]}$ (resp., $\d{\bar{x}_3  \in [x_3^\ast, \min (\bar{x}_{30},x_{3f}))}$) with $\bar{x}_{30}=x_{30}-\frac{x_{40}^2}{2b}\mathrm{sign} x_{40} $, we have
\begin{equation} \label{gammafx3}
\gamma_f(\bar{x}_3)=\gamma(t_f(\bar{x}_3))=\arctan x_2(t_f(\bar{x}_3))/x_1(t_f(\bar{x}_3)).
\end{equation}
If we have 
\begin{equation} \label{dgammadx3}
\begin{split}
\frac{\partial \gamma_f(\bar{x}_3)} {\partial \bar{x}_3} 
&= \frac{1}{v} \left( \frac{\partial x_2(t_f(\bar{x}_3))} {\partial \bar{x}_3} \cos \gamma_f - \frac{\partial x_1(t_f(\bar{x}_3))} {\partial \bar{x}_3} \sin \gamma_f \right) \\
&=\frac{1}{v_f t_f}  \int_0^{t_f}\left( a \big( T_1 \sin(\bar{x}_3-\gamma_f)+ t_f \cos(\bar{x}_3-\gamma_f)\big) - g_0 T_1\cos \gamma_f \right) dt \\
&=\frac{ 1}{v_f t_f} \int_0^{t_f} \left( a \sqrt{T_1^2+t_f^2} \sin(\bar{x}_3-\gamma_f+\bar{\varphi})- g_0 T_1 \cos \gamma_f \right) dt > 0,
\end{split}
\end{equation}
where
\begin{equation*}
\begin{split}
 v_f &= \sqrt{x_1(t_f(\bar{x}_3))^2+x_2(t_f(\bar{x}_3))^2},\quad
\bar{\varphi} = \arctan \left( \frac{t_f}{T_1} \right), \\
 T_1&=\frac{1}{\sqrt{(\bar{x}_3-x_{30})/b+x_{40}^2/(2b^2)}}+\frac{1}{\sqrt{(\bar{x}_3-\gamma_f)/b}},
\end{split}
\end{equation*}
for all $x(t)$ in $\mathcal{F}_{x_3}$, then we have that $\gamma_f(\bar{x}_3)$ is monotone with $\bar{x}_3$. Therefore, the value of $\gamma_f(\bar{x}_3)$ reaches its maximum (resp., minimum) when $\bar{x}_3 = x_3^\ast$. In this sense, we have a reachable set of $\gamma_f$ as a function of $\bar{x}_3$. 

\begin{rem} \label{rem_casedist}
In the anticlockwise case, the trajectories generally take the form of $A_+A_-A_+$. However, if the condition \eqref{dgammadx3} is valid, $\gamma_f(\bar{x}_3)$ achieves a minimum extremal value over $[\max(\bar{x}_{30},x_{3f}),x_3^\ast]$ when $\bar{x}_3=\max(\bar{x}_{30},x_{3f})$. Then, if $\gamma_f < \gamma_f(x_{3f})$, the trajectory takes the form $A_-A_+A_-$. There exists a $\hat{x}_3=\bar{x}_3 \in [x_3^\ast,\min(\bar{x}_{30},x_{3f}))$ such that $\gamma_f(\hat{x}_3)=\gamma_0$. Hence, $\bar{x}_3$ takes value in $\mathcal{D}_{x_3}^{ac}=(\max(\bar{x}_{30},x_{3f}),x_3^\ast] \cup [\hat{x}_3,\min(\bar{x}_{30},x_{3f}))$ for anticlockwise maneuvers. For the clockwise maneuvers, we have that $\bar{x}_3$ takes value in $\mathcal{D}_{x_3}^{c}=(\max(\bar{x}_{30},x_{3f}), \hat{x}_3] \cup [x_3^\ast, \min(\bar{x}_{30},x_{3f}))$ where $\hat{x}_3 \in (\max(\bar{x}_{30},x_{3f}),x_3^\ast]$ being the extremal value such that $\gamma_f(\hat{x}_3)=\gamma_0$.
\end{rem}

The positivity condition \eqref{dgammadx3} is hard to check explicitly, however, numerically this condition can be verified easily for given terminal conditions. 
This is why we take it as an assumption. Accordingly, we make the following assumptions throughout this section. The first assumption is that $\tau_1$, $\tau_2$, $\tau_3$ and $t_f$ are nonnegative real numbers. The second assumption is that \eqref{dgammadx3} holds. The third one is that the spacecraft would not crash after the maneuver. The results of our numerical simulations are consistent with these assumptions:
\begin{itemize}
\item the real numbers $x_{30}$, $x_{40}$, $x_{3f}$ are chosen such that $\tau_1 \geq 0$, $\tau_2 \geq 0$, $\tau_3 \geq 0$ and $t_f >0$;
\item for every $ \bar{x}_3 \in (\max(\bar{x}_{30},x_{3f}),x_3^\ast]$ (resp., $\bar{x}_3  \in [x_3^\ast, \min (\bar{x}_{30},x_{3f}))$) with $\bar{x}_{30}=x_{30}-\frac{x_{40}^2}{2b}\mathrm{sign} x_{40} $, we have
\begin{align*}
\int_0^{t_f} \sin(\bar{x}_3-\gamma_f+\bar{\varphi}) \, dt > \frac{g_0 t_f \cos \gamma_f}{a \sqrt{1+\tan^2\bar{\varphi}} } ;
\end{align*}
\item $x_1(t_f)>0$, $x_2(t_f)>0$.
\end{itemize}
Under these assumptions, we have the following chattering prediction result.

\begin{thm}[Chattering prediction]\label{thm_predic}
Let $x(\cdot) \in \mathcal{F}_{x_3}$ be an optimal trajectory of $\MTTP$ in the flat-Earth case. In the anticlockwise case (resp., in the clockwise case), if 
\begin{equation} \label{pitch_cond}
	S_C \geq 0\qquad (\textrm{resp., if}\ S_C \leq 0),
\end{equation}
with $S_C$ defined by
\begin{equation} \label{Scdef}
    S_C = x_2(t_f(x_3^\ast))-x_1(t_f(x_3^\ast))\tan \gamma_{f},
\end{equation}
where
$x_1(t_f(x_3^\ast)) = x_{10} +\int_0^{t_f(x_3^\ast)} a \cos x_3(t) \, dt$, $x_2(t_f(x_3^\ast)) = x_{20} +\int_0^{t_f(x_3^\ast)} (a \sin x_3(t) -g_0) \, dt$, and $x_3(t)$ is calculated from \eqref{pitchup_x3} with $\bar{x}_3=x_3^\ast$, then $x(\cdot)$ does not involve any singular arc.
\end{thm}

\begin{proof} 
In the conterclockwise case, if $S_C \geq 0$, then we get from \eqref{gammafx3} and \eqref{Scdef} that $\tan \gamma_f(x_3^\ast) \geq \tan \gamma_f$ provided that $x_1(t_f)>0$ and that $x_2(t_f)>0$. Using that $\gamma_f = x_{3f} \in \mathcal{D}_f$, it follows that
$
 (\gamma_f(x_3^\ast)-\gamma_0) \geq (\gamma_f -\gamma_0).
$
Since $\partial \gamma_f(\bar{x}_3) / \partial \bar{x}_3 > 0$, we infer from the implicit function theorem that there exists a $\bar{x}_3=\mathcal{X}(\gamma_f) \in \mathcal{D}_{x_3}^{ac}$, where $\mathcal{X}(\cdot)$ is $C^1$ function, such that $\gamma_f(x_3^\ast) \geq \gamma_f(\bar{x}_3)$ and that the corresponding trajectory $x(t)$ is an optimal trajectory for $\MTTP$ with terminal value of $\gamma_f(\bar{x}_3)$. The proof is similar in the clockwise case.
\end{proof}

\begin{rem} \label{rem_predic}
If \eqref{pitch_cond} is not satisfied, then there are two possible types of solutions: one has more bang arcs, the other has a singular arc with chattering arcs around the singular junctions. The points of $S_2$ actually correspond to the zeros of the second-order time derivative of the switching function ($z_3=0$ in the semi-canonical form), and so the zeros of $S_2$ will impose an immediate effect on the switching function, but ensure the switching function to have more possible switchings. The numerical results show that the additional bang arcs lead to extremals that are closer to the singular surface with an exponential speed. 
\end{rem}

\subsection{Non-flat case $c>0$} \label{sec_normcase}
If $c > 0$ then the analysis of the problem $\MTTP$ becomes more complicated, but we are able as well to describe the set of initial data for which optimal trajectories do not have any singular arc, as we will see next.

We assume that the condition \eqref{dgammadx3} still holds, i.e.,
$
\frac{\partial \gamma_f(\bar{x}_3)} {\partial \bar{x}_3} >0,
$
and the real numbers $x_{30}$, $x_{40}$, $x_{3f}$ are chosen such that $\tau_1 \geq 0$, $\tau_2 \geq 0$, $\tau_3 \geq 0$ and $t_f >0$.
Assume moreover that the real numbers $v_0$ and $\gamma_0$ are chosen such that the two components of the velocity are positive along the whole trajectory, i.e., $x_1(t)>0$, $x_2(t)>0$, $t\in [0,t_f]$. Using Table \ref{sim_param}, we have 
\begin{equation*}
\begin{split}
&\dot{x}_1(t)  \in  [a \cos x_3 - c v_{max}^2, a \cos x_3], \\
&\dot{x}_2(t)  \in [a \sin x_3 - g_0, a \sin x_3 - g_0 + c v_{max}^2], \\
&\dot{x}_3(t)  \in [x_4 - c v_{max}, x_4],
\end{split}
\end{equation*}
where $v_{max}^2 \approx (x_{10}+a T)^2+(x_{20}+a T + a^2 c T^3/3)^2 < v_m^2$. It can be seen that the terms in $c$ in the dynamics cause a decrease of $x_1$ and $x_3$, and an increase of $x_2$. 
We consider the auxiliary problem
\begin{equation*} 
\left\{\begin{split}
& \min t_f \\
& \dot{x}_1= a \cos x_3-c_1 x_1 x_2, \quad 
\dot{x}_2= a \sin x_3-g_0+c_1x_1^2,\quad 
\dot{x}_3=x_4-cv_{max},\quad 
\dot{x}_4=bu, \\
& x(0)=x_0,\quad x_3(t_f)=x_{3f}, \quad x_4(t_f)=0,
\end{split}\right.
\end{equation*}
where $c, c_1\in[0,10^{-6}]$. Similarly to the flat-Earth case, the solutions of this problem, of the form $A_+A_-A_+$ (resp., $A_-A_+A_-$), can be obtained by integrating the dynamical system, by using the control 
\begin{equation*}
u(t) = \begin{cases}
 +1, t\in [0,\tilde{\tau}_1),\quad (\textrm{resp.}, -1)\\
 -1, t\in [\tilde{\tau}_1,\tilde{\tau}_2)\cup[\tilde{\tau}_2,\tilde{\tau}_3), \quad (\textrm{resp.}, +1)\\
 +1, t\in [\tilde{\tau}_3,\tilde{t}_f], \quad (\textrm{resp.}, -1)
\end{cases}
\end{equation*}
with
\begin{equation*} \label{}
\begin{split}
&\tilde{\tau}_1= - \frac{(x_{40}-c v_{max})}{b} 
	      + \sqrt{\frac{(x_{40}-c v_{max})^2}{2b^2} 
	        - \frac{x_{30} - \bar{x}_3}{b}},\quad
\tilde{\tau}_2  = 2 \tau_1 + \frac{(x_{40}-c v_{max})}{b}, \\
&\tilde{\tau}_3 = \tau_2 + \sqrt{\frac{(-c v_{max})^2}{2b^2} - \frac{x_{3f} - \bar{x}_3}{b}},\qquad
\tilde{t}_f = 2 \tau_3 - \frac{c v_{max}}{b} - \tau_2 .
\end{split}
\end{equation*}
\begin{multline*} 
\Big( \textrm{resp.},
\tilde{\tau}_1= - \frac{(x_{40}-c v_{max})}{b} 
	      + \sqrt{\frac{(x_{40}-c v_{max})^2}{2b^2} 
	        + \frac{x_{30} - \bar{x}_3}{b}},\quad
\tilde{\tau}_2  = 2 \tau_1 + \frac{(x_{40}-c v_{max})}{b}, \\
\tilde{\tau}_3 = \tau_2 + \sqrt{\frac{(-c v_{max})^2}{2b^2} + \frac{x_{3f} - \bar{x}_3}{b}},\qquad
\tilde{t}_f = 2 \tau_3 - \frac{c v_{max}}{b} - \tau_2 \Big)
\end{multline*}
Let $\tilde{\gamma}_f(x_3^\ast,c,c_1)=\min \big(\gamma(t_f(x_3^\ast),c>0,c_1=0),\gamma(t_f(x_3^\ast),c>0,c_1>0)\big)$ for this problem. Based on the numerical results, we make the following assumptions:
\begin{itemize}
\item $\tilde{\gamma}_f(x_3^\ast,c,c_1) \geq \gamma_f(x_3^\ast)=\gamma_f(t_f(x_3^\ast),c=0,c_1=0)$ in the anticlockwise maneuvers;
\item $\tilde{\gamma}_f(x_3^\ast,c,c_1) \leq \gamma_f(x_3^\ast)=\gamma_f(t_f(x_3^\ast),c=0,c_1=0)$ in the clockwise maneuvers;
\end{itemize}
Under these assumptions, we have the following result.

\begin{cor} \label{cor_predic}
Let $x(\cdot)$ be an optimal trajectory of $\MTTP$ in the non-flat case. Then:
\begin{itemize}
\item for an anticlockwise maneuver, if \eqref{pitch_cond} holds true then $x(\cdot)$ does not have any singular arc;
\item for a clockwise maneuver, if 
$$
\tilde{S}_C= x_2(\tilde{t}_f(x_3^\ast),c,c_1)-x_1(\tilde{t}_f(x_3^\ast),c,c_1)\tan \gamma_{f} \leq 0,
$$
where $x_2(\tilde{t}_f(x_3^\ast),c,c_1)/x_1(\tilde{t}_f(x_3^\ast),c,c_1)=\tan\tilde{\gamma}_f(x_3^\ast,c,c_1)$,
then $x(\cdot)$ does not have any singular arc.
\end{itemize}
\end{cor}

\begin{proof}
For an anticlockwise maneuver (resp. a clockwise maneuver), we have that if $S_C \geq 0$ (resp. $\tilde{S}_C \leq 0$), then
$
0 \leq \gamma_f 
\leq  \gamma_f(x_3^\ast) 
\leq \tilde{\gamma}_f(x_3^\ast,c,c_1) , 
$
(resp., $ 0 \geq  \gamma_f \geq  \tilde{\gamma}_f(x_3^\ast,c,c_1)$),
and thus there exists a $\bar{x}_3=\tilde{\mathcal{X}}(\gamma_f)$ such that
$
\big(\gamma_f(\bar{x}_3,c>0,c_1>0) - \gamma_0 \big)  \leq \big( \tilde{\gamma}_f(x_3^\ast,c,c_1) -\gamma_0 \big),
$ 
(resp., $
\big(\gamma_f(\bar{x}_3,c>0,c_1>0) - \gamma_0 \big)  \geq \big( \tilde{\gamma}_f(x_3^\ast,c,c_1) -\gamma_0 \big)
$),
and its associated trajectory is an optimal solution of the problem $\MTTP$. 
\end{proof}

\begin{rem} \label{rem_predi_c}
Similarly to Remark \ref{rem_predic}, in the non-flat case, numerical results show that if the conditions in Corollary \ref{cor_predic} are not satisfied, then the trajectories will have more bang arcs until the singular arc finally appear with chattering type junctions. 
\end{rem}

\section{Numerical Results} \label{sec_NumeSimu}
In this section, we compute numerical optimal strategies, for different initial conditions, either by means of a direct method, or by means of an indirect one (shooting method).
It is important to note that, if the optimal trajectory involves a singular arc and thus has chattering, then the shooting method fails in general. Indeed, the infinite number of switchings may cause a failure in the numerical integration of the dynamical system, and then direct methods may therefore be more appropriate to approach chattering. However, since they are based on a discretization, they can only provide a sub-optimal solution of the problem, having a finite number of switchings.

In the first subsection, we provide several numerical simulations, where the optimal solutions are computed by means of a shooting method, in situations where the optimal trajectory is known to be bang-bang, without any singular arc, and with a finite number of switchings.

In the second subsection, we describe in more details sub-optimal strategies, and we provide evidence of their relevance in cases where we have chattering.

In our numerical simulations, we consider the initial and final conditions settled in Table \ref{cond_init_fina}.
 
\begin{table}[h]
\centering
\begin{tabular}{|c|c|c|c|c||c|c|c|}
  \hline
   & $x_{30}$ & $x_{40}$ & $x_{10}$ & $x_{20}$ & $x_{3f}$ & $x_{40}$ & $x_{2f}-x_{1f}\tan x_{3f}$ \\
  \hline
  Counter-clockwise & $1.3$ & $0.0$ & $v_0\cos x_{30}$ & $v_0\sin x_{30}$ & $1.5$ & $0.0$ & $0.0$ \\
  \hline
  Clockwise         & $1.5$ & $0.0$ & $v_0\cos x_{30}$ & $v_0\sin x_{30}$ & $1.3$ & $0.0$ & $0.0$ \\
  \hline
  \multicolumn{8}{|c|}{$\d{v_0=(x_{10}^2+x_{20}^2)^{1/2}}$} \\
  \hline
\end{tabular}
\caption{Initial and final conditions.}\label{cond_init_fina}
\end{table}

Here, we set $\gamma(0)=\gamma_0=x_{30}$, meaning that before the maneuver the spacecraft was on a trajectory with angle of attack equal to zero. 

Recall that when the optimal trajectory contains a singular arc, then the extremal is normal, i.e. $p^0 \ne 0$ (see Lemma \ref{lem_singular}).  
Moreover, in the flat-Earth case, we have seen from the analysis in Section \ref{sec_limicase} that the bang-bang extremals are normal in case of two control switchings. The argument was based on equations \eqref{p30} and \eqref{p40}. 
Furthermore, it is not difficult to see that if the control switches at least two times, then the extremals are normal. Therefore, abnormal extremals may only occur whenever the control switches at most one time. 

In the non-flat case, since $c>0$ is very small, we can assume that $p^0 <0$ though the abnormal extremals may also exist with a few certain terminal conditions.
Thus, the adjoint vector can be normalized by $p^0=-1$. The results of the numerical simulations are consistent with this assumption.

\subsection{Chattering prediction}\label{sec_NumeChatPred}
In practice, the terminal condition that can take very different values is the initial modulus of velocity $v_0$. Hence, we next investigate the influence of $v_0$ on the occurrence of optimal singular arcs.

\paragraph{Flat-Earth case with two switchings.}
If we consider $v_0$ as variable and if we take $c = 0$, then, by solving $S_C = 0$, we get $\bar{v}_{up} = v_0 = 1086.2\,m/s$ (resp., $\bar{v}_{down} = v_0 = 1694.3\,m/s$) for anticlockwise maneuvers (resp., for clockwise maneuvers). When $v_0 \leq \bar{v}_{up}$ (resp., $v_0 \leq \bar{v}_{down}$), we have $S_C \geq 0$ for anticlockwise maneuvers (resp., $S_C \leq 0$ for clockwise maneuvers). In this case, according to Theorem \ref{thm_predic}, there is no singular arc in the optimal solution. Moreover, the maneuver times for both maneuvers are the same, i.e., $t_f = 36.5437\, s$.

Using an indirect method (shooting method), we compute the optimal solutions of the problem $\MTTP$, in the absence of a singular arc. Recall that the indirect method does not work when there are chattering arcs. From the prediction above, we should therefore be able to use successfully an indirect method when $v_0\leq\bar{v}_{up}$. We will see in numerical simulations that the indirect method works when the trajectory consists of three bang arcs, but fails otherwise due to chattering.

Figure \ref{compare_1} provides the solutions for two different values of the initial velocity modulus $v_0$ for the anticlockwise case, i.e., $v_0=\bar{v}_{up}=1086.2\,m/s$ (plotted in solid lines) and $v_0=1080\,m/s$ (plotted in dashed lines). Figure \ref{compare_1_clock} shows the solutions of clockwise maneuvers with $v_0=\bar{v}_{down}=1694.3\, m/s$ (plotted in solid lines) and $v_0=1690\, m/s$ (plotted in dashed lines). The red star points represent the touching point of the trajectories with the surface $S_2$ (where $x_3(t)$ touches $x_3^\ast$).
It is shown in Figure \ref{compare_1} that there is no singular arc in the trajectories when $v_0 < \bar{v}_{up}$ (resp., in Figure \ref{compare_1_clock} when $v_0 < \bar{v}_{down}$).
The control switchings two times and the $\bar{x}_3$ associated with the dashed line is smaller than $x_3^\ast=x_{3f}+\pi/2$ (resp. bigger than $x_3^\ast=x_{3f}-\pi/2$). 

\begin{figure}[h] 
        \includegraphics[scale=0.6]{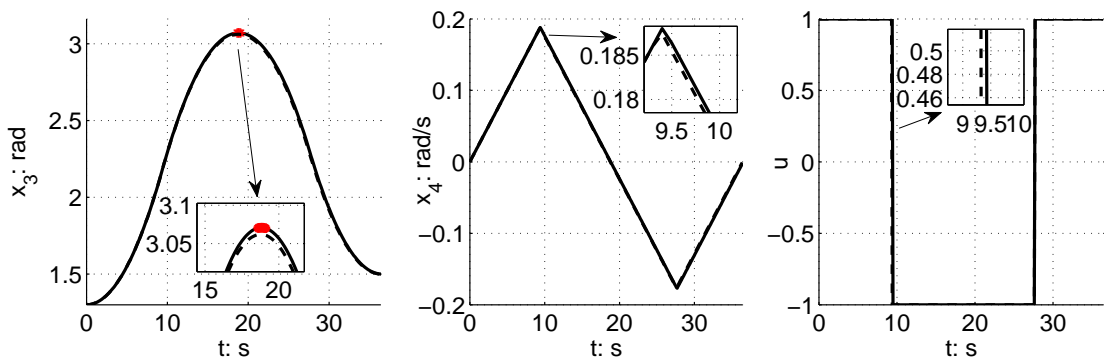}      
        \caption{Time history of $x_3$, $x_4$ and $u$ when $v_0=\bar{v}_{up}$ and $v_0=1080\,m/s$}
        \label{compare_1}
\end{figure}

\begin{figure}[h] 
        \includegraphics[scale=0.6]{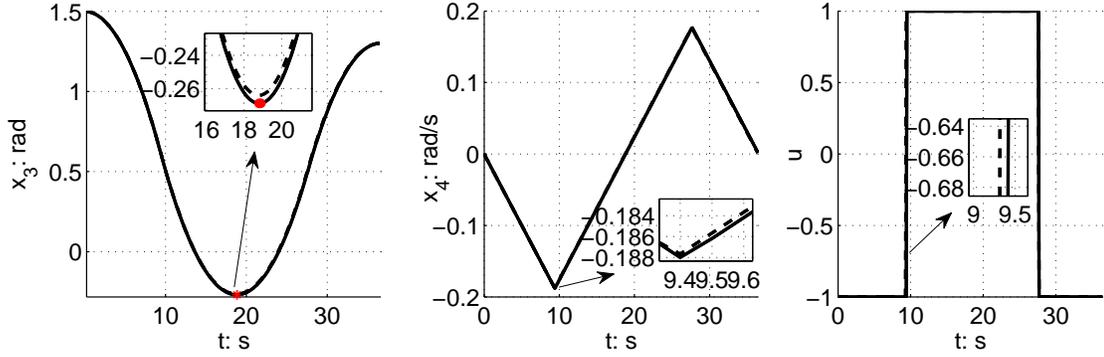}      
        \caption{Time history of $x_3$, $x_4$ and $u$ when $v_0=\bar{v}_{down}$ and $v_0=1690\,m/s$}
        \label{compare_1_clock}
\end{figure}

In Remark \ref{rem_predic}, we mentioned that, when the condition \eqref{pitch_cond} is not satisfied, there are more bang arcs until the appearance of a singular arc. We will show next the solutions with more switchings. As remarked, these results will show that the extremals will get closer to the singular surface $S$ when more bang arcs are present.

\paragraph{Flat-Earth case with more switchings.}
In fact, we can compute the corresponding value of $v_0$ for optimal controls with different number of switchings, in the following way.
Let us assume that the optimal control $u$ has $2m$, $m=1,\cdots,N$ switchings and $u(0)=u_0$ being $+1$ or $-1$, i.e.,
\begin{equation*}
u(t) = \begin{cases}
 u_0, & t\in [\tau_0,\tau_1],\\
-u_0, & t\in [\tau_{4j-3},\tau_{4j-1}], \\
 u_0, & t\in [\tau_{4j-1},\tau_{4j+1}], \\
-u_0, & t\in [\tau_{4m-3},\tau_{4m-1}], \\ 
 u_0, & t\in [\tau_{4m-1},\tau_{4m}],
\end{cases}
\end{equation*}
with $j=1,\cdots,(m-1)$, $t_0=\tau_0$, $t_f=\tau_{4m}$, then we know that $\varphi(\tau_{2k+1})=p_4(\tau_{2k+1})=0$, $k=0,\cdots,2m-1$. Here we have additionally $h_1(\tau_{2m})=p_4(\tau_{2m})=0$, because the maximum $v_0$ corresponding to $2m$ switchings happens when $u$ is about to have one more switchings between $\tau_{2m-1}$ and $\tau_{2m+1}$. 

Let $q = \big( x_3(\tau_{2k}) \big)_{k=1,\cdots,2m-1}$ be the variable vector (of dimension $2m-1$). On Figure \ref{moreswitches} are represented $p_4(t)$, $p_3(t)$, $x_4(t)$ and $x_3(t)$ for an anticlockwise maneuver with $m=3$, the variable $q=(q_1,\cdots,q_5)$ is of dimension $2m-1$.

\begin{figure}[h] 
\centering
        \includegraphics[scale=0.8]{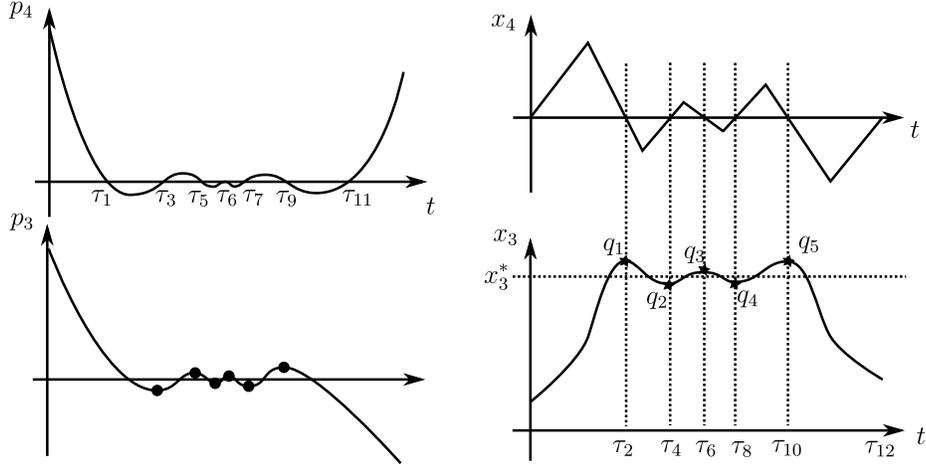}      
        \caption{Example of trajectory associated with optimal control of $6$ switchings.}
        \label{moreswitches}
\end{figure}

Using \eqref{p4t}, we derive $2m-1$ constraints on $q$ without the adjoint vector $p$, i.e.,
\begin{equation} \label{qeqn}
\frac{\tau_{k_1} - \tau_{k_2}}{\tau_{k_3}-\tau_{k_4}} = 
\frac{\int_0^{\tau_{k_1}} \int_0^\tau \cos(x_3(s)-\gamma_f) \,ds\,d\tau-\int_0^{\tau_{k_2}} \int_0^s \cos(x_3(s)-\gamma_f) \,ds\,d\tau}
{\int_0^{\tau_{k_3}} \int_0^\tau \cos(x_3(s)-\gamma_f) \,ds\,d\tau - \int_0^{\tau_{k_4}} \int_0^\tau \cos(x_3(s)-\gamma_f) \,ds\,d\tau},
\end{equation}
where $k_1,k_2,k_3,k_4 \in \{2k+1\mid k=0,\cdots,2m-1\} \cup \{2m\}$ and $k_1 \neq k_2$, $k_3 \neq k_4$. Note that at least one of these equations must involve $\tau_{2m}$.

Since $x_3(\tau_{2k})$, $k=1,\cdots,2m-1$ are local extrema, we must have $x_4(\tau_{2k})=0$, $k=1,\cdots,2m-1$. By integrating the system from $x(0)=x_0$ and requiring  that
\begin{align*}
& x_4(\tau_{2k})=0, \quad k=1,\cdots,2m-1,\\
& x_3(\tau_{2k})=q, \quad k=1,\cdots,2m-1,\\
& x_3(t_f)=x_{3f},\quad x_4(t_f)=0,
\end{align*}
we can parametrize the $\tau_k$, $k=1,\cdots,4m$ by $q$, and hence as well the trajectories $x_3(t)$ and $x_4(t)$ which are parametrized by $\tau_k$, $k=1,\cdots,4m$. More precisely, we have
\begin{align*}
& \tau_1 = -\frac{x_{40}}{b} + \sqrt{\frac{x_{40}^2}{2b^2}+\frac{|q(1)-x_{30}|}{b}},\quad
  \tau_2 =  \tau_1 + \sqrt{\frac{x_{40}^2}{2b^2}+\frac{|q(1)-x_{30}|}{b}},\\
& \tau_{2k+1} =  \tau_{2k} + \sqrt{ \frac{|q(k)-q(k+1)|}{b}},\quad
  \tau_{2k+2} =  \tau_{2k+1} + \sqrt{ \frac{|q(k)-q(k+1)|}{b}},\quad k=1,\cdots,2m-2,\\
& \tau_{4m-1} = \tau_{4m-2} + \sqrt{\frac{|q(2m-1)-x_{3f}|}{b}},\quad
  \tau_{4m}   =  \tau_{4m-1} + \sqrt{\frac{|q(2m-1)-x_{3f}|}{b}}.
\end{align*}
Hence, we can get the value of $q$ by solving \eqref{qeqn}. Then taking $v_0$ as variable and $\gamma(t_f)=\gamma_f$ as shooting function, we can derive the maximum $v_0$ that can be used when we expect the control to have $2m-1$ switchings. 

Using this method, we get that, in the anticlockwise case, when $v_0 \in (\bar{v}_{up},1183.4 ]\, m/s$, the control $u(t)$ has two switchings. When $v_0 \in (1183.4,1999.3]\, m/s$, the control $u(t)$ has four switchings. Then when $v_0 \in (1999.3,2132.1]\,m/s$, the control $u(t)$ has six switchings.  

\begin{figure}[h] 
\centering
        \includegraphics[scale=0.45]{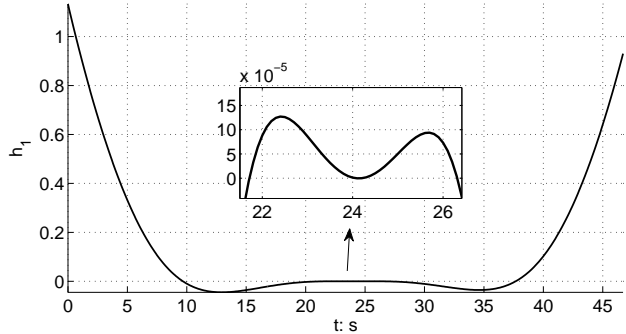}      
        \caption{Switching function $\varphi(t)$ when $v_0=1999.3\,m/s$ in the anticlockwise case.}
        \label{4swiches}
\end{figure}

Figures \ref{4swiches} and \ref{6swiches} give the time history of the switching function $\varphi(t)=h_1(t)$ when $v_0=1999.3\,m/s$ and $v_0=2132.1\,m/s$, respectively. 
Observing from the zoom-in windows of the figures, we see that the switching function is almost equal to zero when $t \in [22,26]\,s$ and $t \in [22,28]\,s$. This implies that the associated extremals are very close to the singular surface $S$ along these time intervals. These figures also show that the additional bang arcs lead rapidly the extremals to get closer to the singular surface $S$ (see Remark \ref{rem_predic}).

\begin{figure}[h] 
\centering
        \includegraphics[scale=0.45]{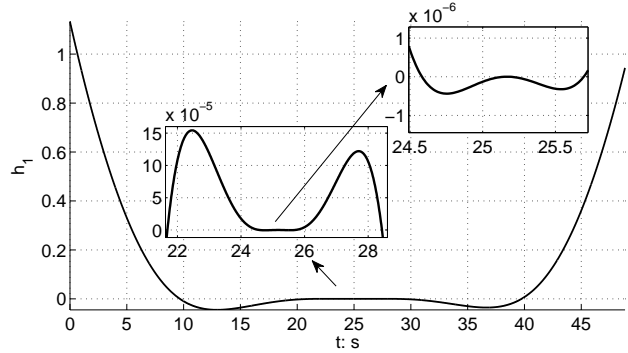}      
        \caption{Switching function $\varphi(t)$ when $v_0=2132.1\,m/s$ in the anticlockwise case.}
        \label{6swiches}
\end{figure}

Note that when $u(t)$ has $2m$ switchings, the trajectory of $x_3(t)$ has between $\max(0,2m-2)$ and $2m$ contact points with the surface $S_2$.
Figure \ref{compare_2} shows the comparison of solutions with $v_0=1350\,m/s$ (solid line) and $v_0=1683\,m/s$ (dashed line). They both belong to the four switchings case, i.e., $m=2$. We can see that the solid line touches the surface $S_2$ two times, while the dashed line touches four times.
\begin{figure}[h] 
        \includegraphics[scale=0.6]{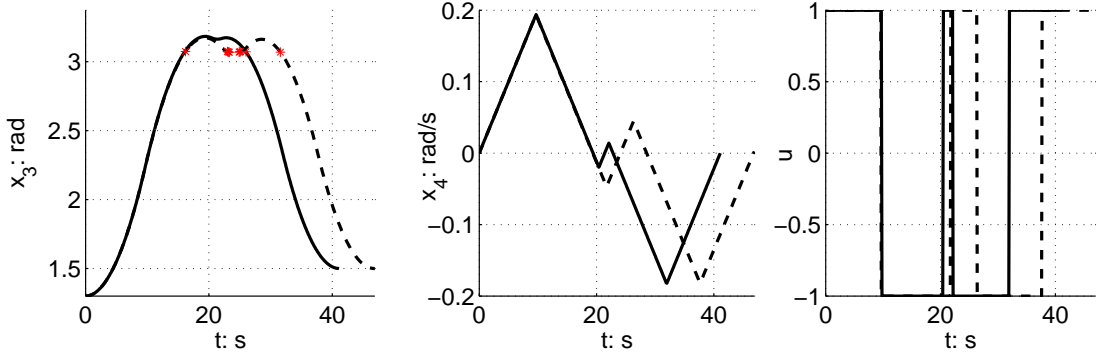}      
        \caption{Time history of $x_3$, $x_4$ and $u$ when $v_0=1350\,m/s$ and $v_0=1683\,m/s$}
        \label{compare_2}
\end{figure}

\medskip
\paragraph{Non-flat case.}
When $c >0$, according to Corollary \ref{cor_predic}, there does not exist any singular arc for anticlockwise maneuvers when $v_0 \leq \bar{v}_{up}$. 
For clockwise maneuvers, if $v_0 \leq \tilde{v}_{down} = 1624.3\,m/s$, then there is no singular arc (this condition is obtained by solving $\tilde{S}_C=0$ with $c=0$ and $c_1=10^{-6}$). The assumptions are also verified.

In Figure \ref{compare_3}, setting $v_0=\bar{v}_{up}$, we compare in anticlockwise case the solution with $c>0$ (plotted with solid line) and the solution with $c=0$ (plotted with dashed line). The trajectory $x_3(t)$ in the flat-Earth case in fact reaches the surface $S_2$ in smaller time than in the non-flat case. 

\begin{figure}[h] 
        \includegraphics[scale=0.6]{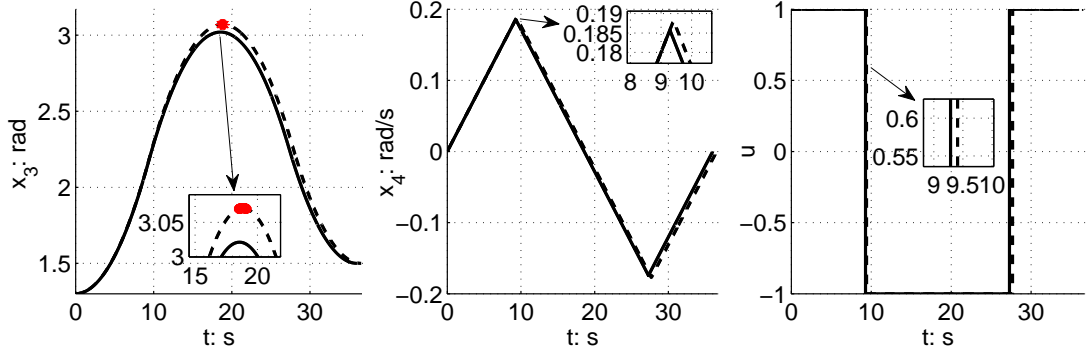}      
        \caption{Time history of $x_3$, $x_4$ and $u$ when $c=0$, $c=10^{-6}$ and $v_0=\bar{v}_{up}$}
        \label{compare_3}
\end{figure}

Let $v_0=\tilde{v}_{down}=1624.3\, m/s$. Figure \ref{compare_4} gives a comparison in the clockwise cases of the solution with $c>0$  (plotted with solid lines) and the solution with $c=0$ (plotted with dashed lines). Both trajectories do not touch the surface $S_2$ and the trajectory in the non-flat case gets ``closer" to $S_2$. The control switchings two times and there is no singular arc.

\begin{figure}[h] 
        \includegraphics[scale=0.5]{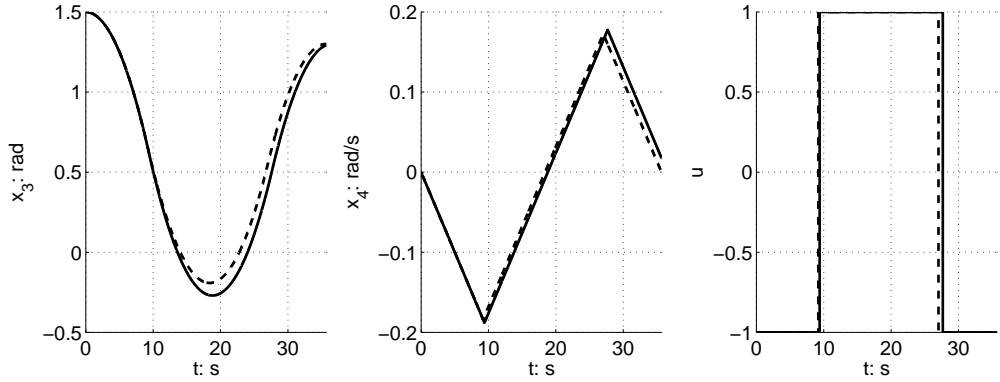}      
        \caption{Time history of $x_3$, $x_4$ and $u$ when $c=0$ and $c=10^{-6}$}
        \label{compare_4}
\end{figure}

In other simulations, we also observe that when $v_0<\tilde{v}_{down}$, the optimal control only has two switchings and $x_3$ will not reach to $x_3^\ast$. However, when $v_0>\tilde{v}_{down}$, new bang arcs appear and the trajectory tends to have chattering arcs. These results illustrate Corollary \ref{cor_predic} and Remark \ref{rem_predi_c}.

\subsection{Sub-optimal strategies}\label{sec_SuboSolu}
Let $N$ be a positive integer.
We consider a subdivision $0=t_0\leq t_1\leq\cdots\leq t_N=t_f$ of the interval $[0,t_f]$ (where $t_i$ are unknown), and we consider piecewise constant controls over this subdivision, thus enforcing the control to switch at most $N$ times. 
We consider the optimal control problem $\MTTP$ with this restricted class of controls, that we denote by $\MTTP_N$.

Solving this problem provides what we call a \emph{sub-optimal strategy} (with at most $N$ switchings), because the optimal value of $\MTTP_N$ must be less than or equal to the optimal value of $\MTTP$.

By the way, we expect that, $\MTTP_N$ $\Gamma$-converges to $\MTTP$ as $N\rightarrow+\infty$, meaning that, in particular, the optimal value of $\MTTP_N$ converges to that of $\MTTP$. We will come back on this issue later.

As in classical direct methods in optimal control, we propose to solve numerically the problem $\MTTP_N$, where the unknowns are the nodes $t_i$ of the subdivision, and the values $u_i$ of the control over each interval $(t_i,t_{i+1})$.
More conveniently, instead of considering the switching times $t_i$ as unknowns, we consider the durations $t_{i+1}-t_i$ as unknowns. Note that these durations may be equal to $0$.

The control is kept constant along each interval of the subdivision, but in order to discretize the state in a finer way, we consider another (much) finer subdivision to compute the discretized state. 

We solve the resulting optimization problem by using \texttt{IPOPT} (see \cite{IPOPT}) combined with the modeling language \texttt{AMPL} (see \cite{Fourer}). 

\paragraph{Numerical results for anticlockwise maneuvers.}
We consider first the case of anticlockwise maneuvers. Let $v_0=3000\,m/s$. For $N=500$, the numerical optimal solution of $\MTTP_N$ is provided on Figures \ref{control_optimal} and \ref{state_optimal}. 
This simulation provides numerical evidence of the fact that we have a singular arc for $t \in [25.7,28.1] \, s$, with a chattering phenomenon at the junction points with the singular arc (see Figure \ref{control_optimal}, on the right, where a zoom is made on those points). The singular control takes values in $[-0.0016,-0.0013]$.

Moreover, the coordinates $x_3(t)$ and $x_4(t)$ oscillate around $x_3=x_3^{\ast}$ and $x_4=0$ respectively, and the coordinates $x_1(t)$ and $x_2(t)$ oscillate around a straight line in the vicinity of the singular arc of the flat-Earth case. This indicates that the singular arc of the non-flat case does not vary much from that of the flat-Earth case.

\begin{figure}[h] \centering
        \includegraphics[scale=0.6]{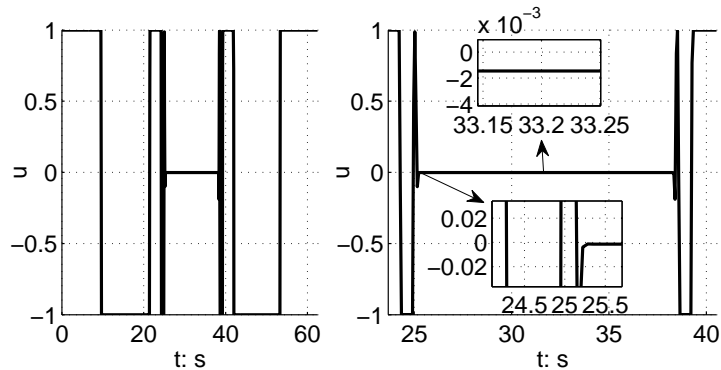}      
        \caption{Control $u(t)$ in anticlockwise maneuver}
        \label{control_optimal}
\end{figure}

\begin{figure}[h] 
	\includegraphics[scale=0.6]{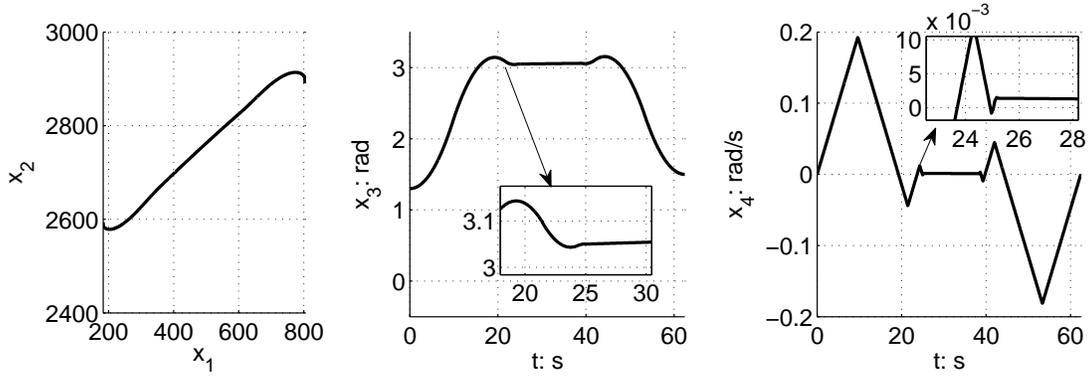}
	\caption{State variable $x(t)$ in anticlockwise maneuver.}
	\label{state_optimal}
\end{figure}

\paragraph{Numerical results for clockwise maneuvers.}
For clockwise maneuvers, still taking $N=500$, the numerical optimal solution of $\MTTP_N$ is provided on Figures \ref{control_optimal_2} and \ref{state_optimal_2}. By comparing the clockwise maneuver in Figure \ref{control_optimal_2} and \ref{state_optimal_2} and the anticlockwise maneuver in Figure \ref{control_optimal} and \ref{state_optimal}, we see that, when $x_4(0)=0$, to realise the same $|\gamma_f-\gamma_0|$, one need $62.43\,s$ for the anticlockwise case and only $59.26\,s$ for the clockwise case.

\begin{figure}[h] \centering
	\includegraphics[scale=0.6]{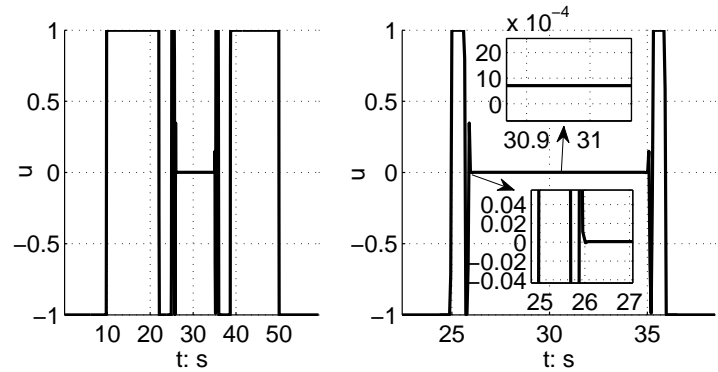}     
	\caption{Optimal control in clockwise maneuver.}
	\label{control_optimal_2}
\end{figure}

\begin{figure}[h]
	\includegraphics[scale=0.6]{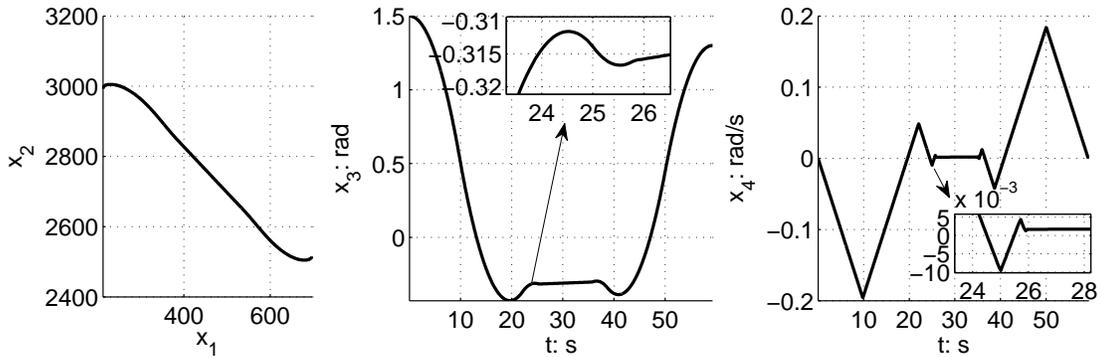}
	\caption{State $x(t)$ in clockwise maneuver.}
	\label{state_optimal_2}
\end{figure}

\paragraph{$\Gamma$-convergence of $\MTTP_N$ towards $\MTTP$.}
It seems natural to expect that, if $N\rightarrow +\infty$, then the solution of $\MTTP_N$ converges to the solution (if it is unique) of $\MTTP$. At least, $\Gamma$-convergence is expected. Such an analysis is beyond the scope of the present paper, however it is interesting to provide numerical simulations, for an anticlockwise maneuver, with several values of $N$:
\begin{equation*}
	N\in\{6,8,10,12,14,16,18,20,30,40,50,100,200,300,400\}.
\end{equation*}
Figure \ref{fig_uN} provides the numerical optimal control obtained for $\MTTP_N$. We observe that, when $N$ becomes larger, then the optimal control seems to converge to its expected limit, that is the optimal control of $\MTTP$ with a singular arc and chattering.
On Figure \ref{fig_tN}, we have reported the values of the maneuver time, in function of $N$. We observe that they seem decrease exponentially with respect to $N$.
This numerical observation is important because, in practice, this means that it is not necessary to take $N$ too large. Even with quite small values of $N$, the minimal time obtained for $\MTTP_N$ seems to be very close to the minimal time for $\MTTP$. Hence the sub-optimal strategy seems to be a very good solution in practice, to bypass the problems due to chattering.

\begin{figure}[h] 
	\includegraphics[scale=0.55]{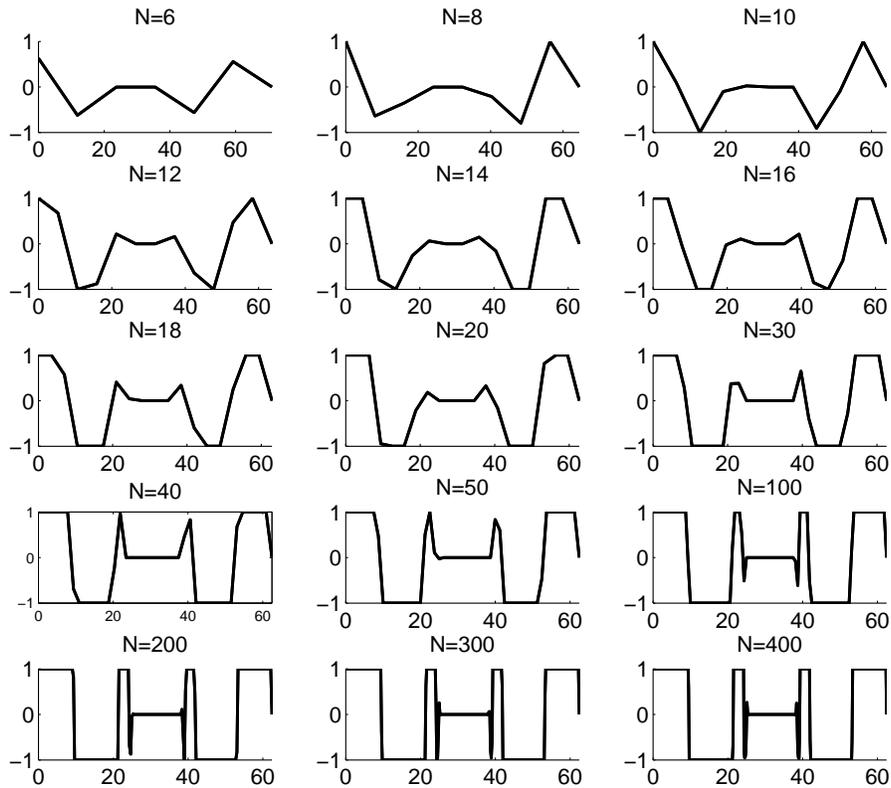}      
	\caption{Control $u(t)$ with different discretization step $N$.}	
	\label{fig_uN}
\end{figure}

\begin{figure}[h] \centering
	\includegraphics[scale=0.6]{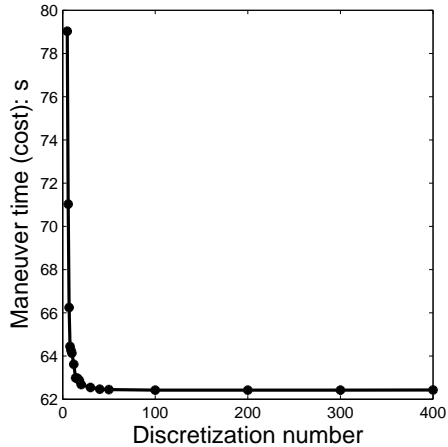}      
	\caption{Maneuver time $t_f$ with respect to the discretization step $N$.}	
	\label{fig_tN}
\end{figure}

We conclude with the following conjecture.

\medskip

\noindent\textbf{Conjecture.}
With obvious notations, we denote by $(x^N(\cdot),u^N(\cdot),t_f^N)$ the optimal solution of $\MTTP_N$, and by $(x(\cdot),u(\cdot),t_f)$ the optimal solution of $\MTTP$ (assuming that they are unique). Then $t_f^N\rightarrow t_f$ exponentially, $x^N(\cdot)\rightarrow x(\cdot)$ in $C^0$-topology, and $u^N(\cdot)\rightarrow u(\cdot)$ in $L^1$-topology, as $N\rightarrow +\infty$.

\begin{rem}
Such convergence properties have been established in \cite{HT,ST}, but for problems not involving any singular arc. Here, the difficulty of establishing such a result (in particular, for the control) is in the presence of an optimal singular arc.
\end{rem}

\begin{rem}
These simulations were done by using hot-restart, that is, by using the solution of the problem $\MTTP_N$ to initialize the problem with a larger value of $N$.
\end{rem}




\end{document}